\documentclass[12pt,a4paper,reqno]{amsart}
\usepackage{amsmath,amsthm,amssymb,mathtools,enumitem,setspace,multirow,diagbox}
\usepackage{a4wide}
\usepackage[margin=2cm]{geometry}

\usepackage[pagebackref=true,colorlinks=true,citecolor=blue,linkcolor=red,urlcolor=black]{hyperref}

\title[The non-commuting, non-generating graph of a non-simple group]{The non-commuting, non-generating graph\\of a non-simple group}

\author{Saul D. Freedman}
\address{ 
{School of Mathematics and Statistics, University of St Andrews, St Andrews, KY16 9SS, UK}}
\email{sdf8@st-andrews.ac.uk}

\newtheorem{thm}{Theorem}[section]
\newtheorem*{thm*}{Theorem}
\newtheorem*{conj*}{Conjecture}

\newtheorem{lem}[thm]{Lemma}
\newtheorem{ques}[thm]{Question}
\newtheorem{prop}[thm]{Proposition}
\newtheorem*{rem*}{Remark}
\theoremstyle{remark}

\theoremstyle{definition}
\newtheorem{rem}[thm]{Remark}
\newtheorem{eg}[thm]{Example}
\newtheorem{assump}[thm]{Assumption}

\newcounter{claim}[thm]

\numberwithin{equation}{section}

\newcommand\diam{\mathrm{diam}}
\newcommand\sd{\mkern1.5mu{:}\mkern1.5mu}

\newcommand\nc{\Xi}
\newcommand\nd{\Xi^+}

\makeatletter
\newcommand{\slteq}{
  \mathrel{\mathpalette\sl@unlhd\relax}
}

\newcommand{\sl@unlhd}[2]{
  \sbox\z@{$#1\lhd$}
  \sbox\tw@{$#1\leqslant$}
  \dimen@=\ht\tw@
  \advance\dimen@-\ht\z@
  \ifx#1\displaystyle
    \advance\dimen@ .2pt
  \else
    \ifx#1\textstyle
      \advance\dimen@ .2pt
    \fi
  \fi
  \ooalign{\raisebox{\dimen@}{$\m@th#1\lhd$}\cr$\m@th#1\leqslant$\cr}
}
\makeatother

\renewcommand{\le}{\leqslant}
\renewcommand{\ge}{\geqslant}
\renewcommand{\trianglelefteq}{\slteq}

\begin{document}

\onehalfspacing

\begin{abstract}
Let $G$ be a (finite or infinite) group such that $G/Z(G)$ is not simple. The non-commuting, non-generating graph $\nc(G)$ of $G$ has vertex set $G \setminus Z(G)$, with vertices $x$ and $y$ adjacent whenever $[x,y] \ne 1$ and $\langle x, y \rangle \ne G$. We investigate the relationship between the structure of $G$ and the connectedness and diameter of $\nc(G)$. In particular, we prove that the graph either: (i) is connected with diameter at most $4$; (ii) consists of isolated vertices and a connected component of diameter at most $4$; or (iii) is the union of two connected components of diameter $2$. We also describe in detail the finite groups with graphs of type (iii). In the companion paper \cite{simplepaper}, we consider the case where $G/Z(G)$ is finite and simple.
\end{abstract}

\maketitle

\section{Introduction}
\label{sec:intro}

Given a group $G$, knowledge about the generating pairs of $G$ and their statistics have found a vast number of applications throughout abstract and computational group theory. Similarly, it is useful to gain information about the pairs of elements of $G$ that \emph{do not} generate the group. If $G$ is non-abelian, then these clearly include all commuting pairs, and so our interest lies in the remaining non-generating pairs.

Information about these pairs of elements is encoded in the \emph{non-commuting, non-generating} graph $\nc(G)$ of $G$, which has vertex set $G \setminus Z(G)$, with two vertices $x$ and $y$ adjacent if and only if $[x,y] \ne 1$ and $\langle x, y \rangle \ne G$. Note that the central elements of $G$ are excluded from the graph's vertex set for convenience, as otherwise they would always be isolated. Aside from this redefinition of the vertex set, the graph $\nc(G)$ is the difference between two consecutive graphs in Cameron's \cite[\S2.6]{camerongraphs} hierarchy of graphs defined on the elements of a group, namely, the \emph{non-generating graph} and the \emph{commuting graph}. The same is true for the \emph{generating graph}, which was introduced in \cite{liebeckshalev2gen}, and for certain graphs studied in \cite{aalipour,powerdiff}.

Many authors (e.g.~\cite{bgk,burnessspread,crestani}) have studied the generating graph, and in particular its connectedness and diameter. For example, Burness, Guralnick and Harper \cite{burnessspread} recently showed that that the generating graph of a finite group is connected if and only if its diameter is at most $2$, and that this occurs precisely when every proper quotient of the group is cyclic.

In \cite{nilppaper}, we showed that the graph $\nd(G)$ induced by the non-isolated vertices of $\nc(G)$ has a similarly small diameter whenever the (finite or infinite) group $G$ is nilpotent, and more generally, when every maximal subgroup of $G$ is normal. In particular, we proved the following theorem, together with detailed structural relationships between $G$ and $\nc(G)$ in the finite case. Note that, in general, $\nc(G)$ has no edge precisely when every proper subgroup of $G$ is abelian.

\begin{thm}[{\cite[Theorem 13]{nilppaper}}]
\label{thm:nilpncgeneral}
Let $G$ be a group with every maximal subgroup normal. If $\nc(G)$ has an edge, then $\nd(G)$ is connected with diameter $2$ or $3$. Moreover, if $\nd(G)$ has diameter $3$, then $\nc(G) = \nd(G)$.
\end{thm}

In this paper, we extend the results of \cite{nilppaper} to the case where $G/Z(G)$ is an arbitrary non-simple group, via the following theorem. Here, an abstract group is \emph{primitive} if it has a core-free maximal subgroup $H$ (so that $G$ acts faithfully and primitively on the right cosets of $H$, which is a point stabiliser for the action). Finally, we denote the diameter of a graph $\Gamma$ by $\diam(\Gamma)$.

\begin{thm}
\label{thm:ncmainsummary}
Let $G$ be a group such that $\overline G := G/Z(G)$ is not simple and $\nc(G)$ has an edge. Then (at least) one of the following holds.
\begin{enumerate}[label={(\roman*)},font=\upshape]
\item \label{ncmainsummary3} $\nc(G)$ has an isolated vertex, and $\diam(\nd(G)) = 2$. If $\overline G$ has a proper non-cyclic quotient, then $G$ is soluble.
\item \label{ncmainsummary4} $\nc(\overline G)$ has an isolated vertex, and $\diam(\nd(G)) \in \{2,3,4\}$. Additionally, $\overline G$ is an infinite, insoluble primitive group with every proper quotient cyclic.
\item \label{ncmainsummary5} $\nc(G)$ is connected with diameter $2$ or $3$.
\item \label{ncmainsummary6} $\nc(G)$ is connected with diameter $4$, $G$ is infinite, and $\overline G$ has a proper non-cyclic quotient.
\item \label{ncmainsummary7} $\nc(G)$ is the union of two connected components of diameter $2$. 
\end{enumerate}
\end{thm}

Hence each component of $\nc(G)$ has diameter at most $4$, and if the graph has more than one nontrivial component (i.e., containing at least two vertices), then $\nc(G)$ is the union of two components of diameter $2$. The former restriction distinguishes $\nc(G)$ from the generating graph of a non-simple finite group, where a component can have arbitrarily large diameter \cite[Theorem 1.3]{crestani}. On the other hand, if $G$ is finite and soluble, then the subgraph of the generating graph induced by its non-isolated vertices has diameter at most $3$ \cite[Theorem 1]{lucchinisoluble}. The same holds for $\nc(G)$, unless Theorem~\ref{thm:ncmainsummary}\ref{ncmainsummary7} applies. We also show in \cite[Proposition 5.9.9]{saulthesis} that if $\overline G$ is finite, $G$ has an abelian maximal subgroup, and $\nc(G)$ has an edge, then $\diam(\nd(G)) = 2$.

Note that it is an open problem to determine whether cases (ii) and (iv) can occur; see Questions~\ref{ques:xz}, \ref{ques:xy} and \ref{ques:imprimnc}, and Remark~\ref{rem:dist4}. If $G$ does satisfy case (ii), then Lemma~\ref{lem:ncinsprimquo} yields further information about the structures of $\nc(G)$ and $\nc(\overline G)$.

In our companion paper \cite{simplepaper}, we explore the diameter of $\nc(G)$ when $G/Z(G)$ is finite and simple; in particular, we prove that $\nc(G)$ is always connected in this case. Hence each finite group $G$ with $\nd(G)$ not connected satisfies Theorem~\ref{thm:ncmainsummary}\ref{ncmainsummary7}. Our second main theorem precisely describes these finite groups. We write $\Phi(H)$ to denote the Frattini subgroup of a group $H$.

\begin{thm}
\label{thm:sumtwo}
Let $G$ be a finite group. Then $\nc(G)$ is the union of two connected components of diameter $2$ if and only if the following all hold:
\begin{enumerate}[label={(\roman*)},font=\upshape]
\item $G = P \sd Q$, where $P$ and $Q$ are nontrivial Sylow subgroups;
\item $Q$ is cyclic and acts irreducibly on $P/\Phi(P)$;
\item $\Phi(P) = Z(P) \not\le Z(G)$; and
\item the unique maximal subgroup of $Q$ is normal in $G$.
\end{enumerate}
\end{thm}

We will observe in Theorem~\ref{thm:solmaxclass} below (see also \cite{adnan}) that conditions (i)--(ii) of Theorem~\ref{thm:sumtwo} hold if and only if the finite group $G$ has exactly two conjugacy classes of maximal subgroups.

The paper is organised as follows. In \S\ref{sec:prelims}, we present preliminary results on maximal subgroups of $G$ and on $\nc(G)$. Next, \S\ref{sec:twoconjmax} focuses on groups whose maximal subgroups satisfy certain conditions, including finite groups with exactly two conjugacy classes of maximal subgroups. In \S\ref{sec:normmaxzm}, we bound distances in $\nc(G)$ when $G$ has a normal, non-abelian maximal subgroup $M$ with $Z(G) < Z(M)$. These bounds are applied in \S\ref{sec:normalcent}, where we prove Theorem~\ref{thm:ncmainsummary} and \ref{thm:sumtwo} when $G/Z(G)$ has a proper non-cyclic quotient. We then complete the proof of our main theorems in \S\ref{sec:finitesol} by considering the remaining groups. In \S\ref{sec:normalcent}--\ref{sec:finitesol}, we also exhibit structural relationships between $G$ and $\nc(G)$ using concrete examples, many of which involve Magma \cite{magma} computations. 

\section{Preliminaries}
\label{sec:prelims}

In this section, we state several preliminary results related to maximal subgroups of a group $G$ and its non-commuting, non-generating graph $\nc(G)$. Given vertices $x$ and $y$ of a graph, $d(x,y)$ denotes their distance in the graph, and we write $x \sim y$ if the vertices are adjacent.

Throughout this paper, we will implicitly use the following proposition.

\begin{prop}[{\cite[Proposition 2.1.1]{cohn}}]
\label{prop:fingenmaxsub}
Suppose that $G$ is finitely generated, and let $H$ be a proper subgroup of $G$. Then $H$ is contained in a maximal subgroup of $G$.
\end{prop}

\begin{thm}[{\cite[Theorems I.4, IV.11, and IV.14]{ore}}]
\label{thm:finsolconjmax}
Suppose that $G$ is finite and soluble, and let $L$ and $M$ be distinct maximal subgroups of $G$. Then the following are equivalent:
\begin{enumerate}[label={(\roman*)},font=\upshape]
\item $L$ and $M$ are conjugate in $G$;
\item $\mathrm{Core}_G(L) = \mathrm{Core}_G(M)$; and
\item $LM \ne G$.
\end{enumerate}
\end{thm}

\begin{lem}
\label{lem:nonabmax}
Let $(X,Y)$ be a pair of proper subgroups of $G$, with $X$ maximal and $Y \not\le X$. If $Z(X) \cap Y \not\le Z(G)$, then $Z(Y) \le X \cap Y$. If, in addition, $X$ is abelian, then $Z(Y) \le Z(G)$.
\end{lem}

\begin{proof}
Let $z \in (Z(X) \cap Y) \setminus Z(G)$. Then $X = C_G(z)$, and hence $Z(Y) \le C_Y(z) = X \cap Y$. If $X$ is abelian, then each element of $Z(Y)$ is centralised by $\langle X, Y \rangle = G$, and hence $Z(Y) \le Z(G)$.
\end{proof}

Our next result generalises an argument used in the proof of \cite[Proposition 10]{nilppaper}.

\begin{lem}
\label{lem:threemaxint}
Let $(W,X,Y)$ be a triple of distinct proper subgroups of $G$, with $X$ and $W \cap X$ normal in $G$, $X$ and $Y$ maximal in $G$, and $W \cap X \not\le Y$. Then $X \cap Y \not\le W$.
\end{lem}

\begin{proof}
Assume for a contradiction that $X \cap Y \le W$. Then $$G/(W \cap X) = (W \cap X)Y/(W \cap X) \cong Y/(W \cap X \cap Y) = Y/(X \cap Y) \cong XY/X = G/X.$$ This contradicts the fact that $G/X$ is simple, while $G/(W \cap X)$ is not.
\end{proof}

\begin{prop}[{\cite[\S2]{nilppaper}}]
\label{prop:ncbasics}
Suppose that $G$ is non-abelian.
\begin{enumerate}[label={(\roman*)},font=\upshape]
\item \label{nodiamonenc} No connected component of $\nc(G)$ has diameter $1$.
\item \label{highgengp} If $G$ is not $2$-generated, then $\nc(G)$ is connected with diameter $2$.
\item \label{isolvertnc} Suppose that $G$ is $2$-generated, and let $g \in G \setminus Z(G)$. Then $g$ is an isolated vertex of $\nc(G)$ if and only if $g$ lies in a unique maximal subgroup $M$ of $G$ and in $Z(M)$. Moreover, if $g$ is not isolated, then $g \in L \setminus Z(L)$ for some maximal subgroup $L$ of $G$.
\item \label{propernc} Let $H$ be a proper non-abelian subgroup of $G$. Then the induced subgraph of $\nc(G)$ corresponding to $H \setminus Z(H)$ is connected with diameter $2$.
\end{enumerate}
\end{prop}

\begin{prop}
\label{prop:isolabmax}
Suppose that $G$ is $2$-generated, and that an element $g \in G$ lies in a unique maximal subgroup $M$ of $G$ and in $Z(M)$. If $G$ is finite or $M \trianglelefteq G$, then $M$ is abelian.
\end{prop}

\begin{proof}
If $M \trianglelefteq G$, then $M$ is abelian by \cite[Proposition 11]{nilppaper}. Assume therefore that $G$ is finite and $M \not\trianglelefteq G$. Then $g \notin Z(G)$ and $M = C_G(g)$. Moreover, $g^h \in M^h \setminus M$ for each $h \in G \setminus M$, yielding $\langle g, g^h \rangle = G$ and $Z(G) = C_G(g) \cap C_G(g^h) = M \cap M^h$. Thus $G/Z(G)$ is a Frobenius group with Frobenius complement $M/Z(G)$, and so $G$ has a nontrivial normal subgroup $N$ with $G = NM$ and $N \cap M = Z(G)$. Observe that $G = \langle N, g \rangle = N \langle g \rangle$, and hence $M$ is the abelian group $\langle Z(G), g \rangle$.
\end{proof}

Our next two results involve the \emph{non-commuting graph} of $G$, which has vertex set $G \setminus Z(G)$, with two vertices adjacent if and only if they do not commute.

\begin{prop}[{\cite[Proposition 2.1]{abdollahi}}]
\label{prop:noncomdiam}
The non-commuting graph of a non-abelian group is connected with diameter $2$.
\end{prop}

\begin{lem}
\label{lem:noncycedge}
Let $N$ be a normal subgroup of $G$, and suppose that $G/N$ is not cyclic. Additionally, let $n \in N$ and $g \in G$. Then $\{n,g\}$ is an edge of $\nc(G)$ if and only if $[n,g] \ne 1$, i.e., if and only if $\{n,g\}$ is an edge of the non-commuting graph of $G$.
\end{lem}

\begin{proof}
Since $G/N$ is not cyclic, it is clear that $\langle N, g \rangle < G$. The result now follows from the definitions of $\nc(G)$ and the non-commuting graph of $G$.
\end{proof}

\begin{lem}
\label{lem:quotientnc}
Let $N$ be a normal subgroup of $G$.
\begin{enumerate}[label={(\roman*)},font=\upshape]
\item \label{quotientnc1} Let $x,y \in G$. Then $\{Nx,Ny\}$ is an edge of $\nc(G/N)$ if and only if $[x,y] \notin N$ and $\langle x, y, N \rangle < G$.
\item \label{quotientnc2} Suppose that $\nc(G/N)$ has a connected component $C$ containing at least two vertices. Then the subgraph of $\nc(G)$ induced by the vertices in $\{x \in G \setminus Z(G) \mid Nx \in C\}$ is connected with diameter at most $\diam(C)$.
\end{enumerate}
\end{lem}

\begin{proof} Observe that $[Nx,Ny]$ is the identity in $G/N$ if and only if $[x,y] \in N$, and $\langle Nx,Ny \rangle < G/N$ if and only if $\langle x, y, N \rangle < G$. Thus we obtain (i).

To prove (ii), let $x, y \in G \setminus Z(G)$ with $Nx, Ny \in C$. As $k:=\diam(C) \ge 2$ by Proposition~\ref{prop:ncbasics}\ref{nodiamonenc}, there exist $g_1,\ldots,g_n \in G \setminus Z(G)$, with $n \le k$ and $g_n = y$, such that $\nc(G/N)$ contains the path $(N x, N g_1, \ldots, N g_n)$. By (i), $(x,g_1,\ldots,g_n)$ is a path in $\nc(G)$ of length $n \le k$.
\end{proof}

\begin{lem}
\label{lem:centremaxmax}
Let $(x,J,K)$ be such that $J$ and $K$ are proper subgroups of $G$, with $x \in J \setminus Z(J)$ and $x \notin K$. In addition, suppose that $H:=J \cap K$ is a maximal subgroup of $J$, or that $K$ is a normal maximal subgroup of $G$.
\begin{enumerate}[label={(\roman*)},font=\upshape]
\item \label{centremaxmax1} There exists $h \in H$ such that $\{x,h\}$ is an edge of $\nc(G)$, and in particular $C_H(x) < H$.
\item \label{centremaxmax2} Suppose that there exists $y \in K \setminus Z(K)$ with $y \notin J$. If $H$ is a maximal subgroup of $K$, or if $J$ is a normal maximal subgroup of $G$, then there exists an element $g \in H$ such that $(x,g,y)$ is a path in $\nc(G)$.
\end{enumerate}
\end{lem}

\begin{proof}
Note that if $K$ is a normal maximal subgroup of $G$, then $H$ is maximal in $J$. Thus we may assume in general that $H$ is maximal in $J$, and similarly, that $H$ is maximal in $K$ in (ii).

Observe that $C_H(x) < H$, as otherwise $\langle H, x \rangle = J$ would centralise $x$. For each $h \in H \setminus C_H(x)$, the subgroup $\langle x,h \rangle$ lies in $J < G$. Hence $x \sim h$, and we obtain (i).

Now suppose that $H$ is maximal in $K$, and let $y$ be as in (ii). Arguing as above, $C_H(y) < H$. There exists $g \in H \setminus (C_H(x) \cup C_H(y))$, as the union of two proper subgroups of $H$ is a proper subset. Furthermore, $\langle x, g \rangle \le J < G$ and $\langle g, y \rangle \le K < G$, so that $x \sim g \sim y$, yielding (ii).
\end{proof}

\begin{lem}
\label{lem:ncmaxnormcombined}
Suppose that $G$ is $2$-generated, and let $(x,L,y,M)$ be such that $L$ and $M$ are non-abelian maximal subgroups of $G$, with $L \trianglelefteq G$, $x \in L \setminus Z(L)$ and $y \in M \setminus Z(M)$. Suppose also that $C_L(x) \trianglelefteq G$ or $M \trianglelefteq G$. Then $d(x,y) \le 3$. Moreover, $d(x,y) = 3$ if and only if either:
\begin{enumerate}[label={(\roman*)},font=\upshape]
\item $x \in Z(M)$, $y \notin L$, and $M$ is the only maximal subgroup of $G$ containing but not centralising $y$; or
\item $y \in Z(L)$, $x \notin M$, and $L$ is the only maximal subgroup of $G$ containing but not centralising $x$.
\end{enumerate}
\end{lem}

\begin{proof} If $M \trianglelefteq G$, then the result is precisely \cite[Lemma 12]{nilppaper}. Assume therefore that $C_L(x) \trianglelefteq G$ and $M \not\trianglelefteq G$. Additionally, let $\{(f,A),(g,B)\} = \{(x,L),(y,M)\}$. We claim that if $f \in Z(B)$, then $G/\langle f \rangle^G$ is not cyclic. This is clear if $B = L$. If instead $B = M$, then $C_L(x) = L \cap C_G(x) = L \cap M$, and so $\langle x \rangle^G \le L \cap M$. Furthermore, $M/(L \cap M) \not\trianglelefteq G/(L \cap M)$, since $M \not\trianglelefteq G$. Hence $G/(L \cap M)$ is not cyclic, and it follows that $G/ \langle x \rangle^G$ is also not cyclic, as claimed.

We split the remainder of the proof into four cases, corresponding to where $x$ lies with respect to $M$ and $Z(M)$ and where $y$ lies with respect to $L$ and $Z(L)$.

\medskip

\noindent \textbf{Case (a)}: $x \in M \setminus Z(M)$ or $y \in L \setminus Z(L)$. Here, we obtain $d(x,y) \le 2$ from Proposition~\ref{prop:ncbasics}\ref{propernc}.

\medskip

\noindent \textbf{Case (b)}: $x \notin M$ and $y \notin L$. Since $x$ lies in $L$ and in $C_G(x)$ but not in $M$, we see that $C_G(x) \cap L \not\le M$. Applying Lemma~\ref{lem:threemaxint} to $(C_G(x),L,M)$ therefore gives $L \cap M \not\le C_G(x)$. Hence $C_{L \cap M}(x) < L \cap M$, and applying Lemma~\ref{lem:centremaxmax} to $(y,M,L)$ yields $C_{L \cap M}(y) < L \cap M$. Thus there exists $h \in L \cap M$ that centralises neither $x$ nor $y$. It follows that $x \sim h \sim y$, and so $d(x,y) \le 2$.

\medskip

\noindent \textbf{Case (c)}: $x \in Z(M)$ and $y \in Z(L)$. Here, $[x,y] = 1$. As the non-commuting graph of $G$ is connected with diameter $2$ by Proposition~\ref{prop:noncomdiam}, this graph contains the path $(x,r,y)$ for some $r \in G \setminus Z(G)$. In addition, $G/\langle x \rangle^G$ and $G/\langle y \rangle^G$ are non-cyclic, by the first paragraph of the proof. It follows from Lemma~\ref{lem:noncycedge} that $(x,r,y)$ is also a path in $\nc(G)$, and hence $d(x,y) = 2$.

\medskip

\noindent \textbf{Case (d)}: $x \in Z(M)$ and $y \notin L$, or $y \in Z(L)$ and $x \notin M$. Here, $f \in Z(B)$ and $g \notin A$, where $\{(f,A),(g,B)\} = \{(x,L),(y,M)\}$ as above. We claim that $C_{A \cap B}(g) < H:=A \cap B$. Indeed, if $B = M$, then applying Lemma~\ref{lem:centremaxmax}\ref{centremaxmax1} to $(g,B,A)$ yields the claim. Otherwise, the claim follows by applying Lemma~\ref{lem:threemaxint} to $(C_G(g),B,A)$, as in the proof of Case (b). In general, as $f \in H \setminus Z(A)$, we see that $Z(A) \cap B < H$. Thus there exists $k \in H \setminus (C_H(g) \cup C_H(A))$. Observe that $g \sim k$, while $d(f,k) \le 2$ by Proposition~\ref{prop:ncbasics}\ref{propernc}. Hence $d(f,g) \le 3$.

It remains to show that $d(f,g) = 3$ if and only if $B$ is the unique maximal subgroup of $G$ that contains but does not centralise $g$. If $B$ is the unique such maximal subgroup, then $B$ contains the neighbourhood of $g$ in $\nc(G)$, while no element of $B$ is a neighbour of $f \in Z(B)$. Thus $d(f,g) > 2$, and so $d(f,g) = 3$ by the previous paragraph.

If instead $g \in K \setminus Z(K)$ for some maximal subgroup $K \ne B$ of $G$, then $K \cap B$ and $C_K(g)$ are proper subgroups of $K$. Hence there exists $s \in K \setminus (B \cup C_K(g))$, and in particular, $s \sim g$. Additionally, since $f \in Z(B)$, the quotient $G/\langle f \rangle^G$ is non-cyclic, by the first paragraph of the proof. As $s \notin B = C_G(f)$, it follows from Lemma~\ref{lem:noncycedge} that $f \sim s$, and thus $d(f,g) \le 2$.
\end{proof}

\section{Groups with two conjugacy classes of maximal subgroups}
\label{sec:twoconjmax}

Here, we consider groups whose maximal subgroups satisfy certain conditions, and in particular, finite groups with exactly two conjugacy classes of maximal subgroups.

\begin{lem}
\label{lem:twoconjclass}
Suppose that $G$ is finitely generated and non-cyclic. Moreover, assume that $G$ contains a normal maximal subgroup $M$, and that $K \cap M = L \cap M$ for all maximal subgroups $K$ and $L$ of $G$ distinct from $M$. Then $K \cap M = K \cap L = \Phi(G)$. Moreover, if $G$ is finite, then $G$ is soluble. If, in addition, $M$ is the unique normal maximal subgroup of $G$, then $G$ contains exactly two conjugacy classes of maximal subgroups.
\end{lem}

\begin{proof}
Let $K$ and $L$ be distinct maximal subgroups of $G$ that are not equal to $M$ (these exist as $G$ is finitely generated and not cyclic). As $K \cap M = L \cap M$, we observe that $K \cap M \le K \cap L < K$. Moreover, $K \cap M$ is maximal in $K$ (by the normality of $M$), and thus $K \cap M = K \cap L$. Hence $K \cap M$ is the intersection of each pair of distinct maximal subgroups of $G$, and so $K \cap M = \Phi(G)$.

We assume from now on that $G$ is finite. Then $|K| = |L|$, and so the set $S$ of orders of maximal subgroups of $G$ has size at most $2$. Suppose first that $G$ is insoluble. Since $|S| \le 2$, the quotient $G/\Phi(G)$ is isomorphic to $H:=(C_2^{3i} \sd \mathrm{PSL}(2,7)) \times C_7^j$, where $i$ and $j$ are non-negative integers \cite{shi}. As $\mathrm{PSL}(2,7)$ contains maximal subgroups of index $7$ and $8$, we deduce from the simplicity of that group that $G$ contains non-normal maximal subgroups $A$ and $B$ of index $7$ and $8$, respectively, contradicting the requirement $|A| = |B|$.

Hence $G$ is soluble. Assume now that $M$ is the unique normal maximal subgroup of $G$. Then $K \not \trianglelefteq G$, and $\mathrm{Core}_G(K) = K \cap M = \Phi(G)$, for each maximal subgroup $K \ne M$. Thus Theorem~\ref{thm:finsolconjmax} shows that $G$ has exactly two conjugacy classes of maximal subgroups.
\end{proof}

We now examine the finite groups satisfying the final conclusion of the previous lemma.

\begin{thm}
\label{thm:solmaxclass}
Suppose that $G$ is finite. Then the following statements hold.
\begin{enumerate}[label={(\roman*)},font=\upshape]
\item \label{solmaxclass1} $G$ contains exactly two conjugacy classes of maximal subgroups if and only if:
\begin{enumerate}[label={(\alph*)},font=\upshape]
\item $G = P \sd Q$, where $P$ and $Q$ are nontrivial Sylow subgroups; and
\item $Q$ is cyclic and acts irreducibly on $P/\Phi(P)$.
\end{enumerate}
\item \label{solmaxclass2} Suppose that \textnormal{(i)(a)} and \textnormal{(b)} hold, and let $R$ be the unique maximal subgroup of $Q$. Then:
\begin{enumerate}[label={(\alph*)},font=\upshape]
\item \label{solmaxclass2a} the maximal subgroups of $G$ are $M:=P
R$ and the conjugates of $\Phi(P)Q$;
\item \label{solmaxclass2b} $R \trianglelefteq G$ if and only if $\Phi(G) = M \cap \Phi(P)Q$;
\item \label{solmaxclass2c} if $R \trianglelefteq G$, then $M = P \times R$, and $\Phi(G) = \Phi(P) \times R$ is the intersection of each pair of distinct maximal subgroups of $G$; and
\item \label{solmaxclass2d} if $R \trianglelefteq G$, $C_M(\Phi(G)) \not\le \Phi(G)$, and $M$ is non-abelian, then $\Phi(G) = Z(M)$, i.e., $\Phi(P) = Z(P)$.
\end{enumerate}
\end{enumerate}
\end{thm}

\begin{proof}
We will begin by proving (i) and (ii)(a). Adnan \cite{adnan} proved that if $G$ contains exactly two conjugacy classes of maximal subgroups, then $G$ satisfies (i)(a) and (b). We will therefore assume that (i)(a) and (b) hold. 
The irreducibility of the action of $Q$ on $P/\Phi(P)$ implies that 
$\langle \Phi(P), Q, x \rangle = G$ for each $x \in P \setminus \Phi(P)$, and so $\Phi(P)Q$ and its $G$-conjugates are maximal subgroups of $G$. Since $R$ is maximal in $Q$ and $P \trianglelefteq G$, we deduce that $M:=PR$ is also a maximal subgroup of $G$. As $G/P$ is cyclic, its subgroup $M/P$ is normal, and hence $M \trianglelefteq G$.

To complete the proofs of (i) and (ii)(a), it suffices to show that we have described all maximal subgroups of $G$. Suppose, for a contradiction, that $G$ contains a maximal subgroup $T$ that is neither equal to $M$ nor conjugate to $\Phi(P)Q$. Then $T$ contains an element $xy$, where $x \in P$ and (without loss of generality) $y$ is a generator for $Q$. For each integer $k$, the projection of $(xy)^k$ onto $Q$ is equal to $y^k$. Thus $|Q|$ divides $|xy|$, and it follows that $T$ contains an element of order $|Q|$. Hence $T$ contains a Sylow subgroup of $G$ of order $|Q|$, and we may assume that $Q \le T$.

Let $S$ be the projection of $T$ onto $P$, i.e., the set of elements $v \in P$ such that there exists $w \in Q$ with $vw \in T$. By the previous paragraph, $S = T \cap P$. Additionally, Theorem~\ref{thm:finsolconjmax} implies that $G = T\Phi(P)Q=TQ\Phi(P)=SQ\Phi(P) = S\Phi(P)Q$. As $G = P \sd Q$, we deduce that $\langle S, \Phi(P) \rangle = P$, hence $S = P$ and $P \le T$. Thus $T$ contains $\langle P, Q \rangle = G$. This contradicts the maximality of $T$, and we obtain (i) and (ii)(a).

We now prove (ii)(b)--(c). Assume first that $\Phi(G) = M \cap \Phi(P)Q$. Then this intersection, which is equal to $\Phi(P)R$, is normal in $G$. Since $R$ is a Sylow $q$-subgroup of $\Phi(P)R$, the Frattini Argument yields $G = \Phi(P)R N_G(R) = \Phi(P)N_G(R)$. Thus $P = \Phi(P)N_G(R) \cap P = \Phi(P)(N_G(R) \cap P)$. Hence $P = N_G(R) \cap P$, i.e., $P \le N_G(R)$. Therefore, $R \trianglelefteq P \sd Q = G$. 

Conversely, assume that $R \trianglelefteq G$. Since $P \cap R = 1$, it is clear that $M = P \times R$. Additionally, as $G=P \sd Q^g$ for each $g \in G$, we see that $$M \cap (\Phi(P)Q)^g = (P \times R) \cap \Phi(P)Q^g = (P \cap \Phi(P))(R \cap Q^g) =  \Phi(P) \times R.$$ As $\Phi(P) \times R$ is the intersection of any two distinct $G$-conjugates of $\Phi(P)Q$, (ii)(b)--(c) hold.

To prove (ii)(d), assume again that $R \trianglelefteq G$. Observe that $C_M(\Phi(G)) = C_{P \times R}(\Phi(P) \times R) = C_P(\Phi(P)) \times R$. Note that $C_P(\Phi(P))Q$ is a subgroup of $G$, since $C_P(\Phi(P))$ is characteristic in $P$. As $\Phi(P)Q$ is the unique maximal subgroup of $G$ containing $Q$, and as $P \cap Q = 1$, it follows that either $C_P(\Phi(P)) \le \Phi(P)$ or $C_P(\Phi(P)) = P$. In the former case, $C_M(\Phi(G)) \le \Phi(P) \times R = \Phi(G)$.

Assume now that $M$ is non-abelian and $C_M(\Phi(G)) \not\le \Phi(G)$. Then $C_M(\Phi(G)) = P \times R = M$ and $\Phi(G) \le Z(M)$. Suppose for a contradiction that $\Phi(G) \ne Z(M)$, so that $Z(M) \not\le \Phi(G) = M \cap \Phi(P)Q$. Then $Z(M) \not\le \Phi(P)Q$, and applying Lemma~\ref{lem:threemaxint} to $(Z(M),M,\Phi(P)Q)$ yields $\Phi(G) = {M \cap \Phi(P)Q} \not\le Z(M)$, a contradiction. Thus $\Phi(G) = Z(M)$, and (ii)(d) follows.
\end{proof}

Note that the above theorem is closely related to Theorem~\ref{thm:sumtwo}. For convenience, we will collect conditions (i)--(iv) of the latter theorem in the following assumption, together with the condition on maximal subgroups that Theorem~\ref{thm:solmaxclass}\ref{solmaxclass1} shows is equivalent to (i)--(ii).

\begin{assump}
\label{assump:22g}
Assume that $G$ is finite and contains exactly two conjugacy classes of maximal subgroups, i.e., that $G = P \sd Q$, where $P$ and $Q$ are nontrivial Sylow subgroups such that $Q$ is cyclic and acts irreducibly on $P/\Phi(P)$. In addition, assume that $\Phi(P) = Z(P) \not\le Z(G)$, and that the unique maximal subgroup of $Q$ is normal in $G$.
\end{assump}

\section{Normal, non-abelian maximal subgroups with large centres}
\label{sec:normmaxzm}

We now focus on the case where $G$ contains a normal, non-abelian maximal subgroup $M$ satisfying $Z(G) < Z(M)$ (equivalently, $Z(M) \not\le Z(G)$). In particular, we determine upper bounds for the distance in $\nc(G)$ between an element of $M \setminus Z(M)$ and an element of $G \setminus M$ or $Z(M) \setminus Z(G)$. We will apply these results in \S\ref{sec:normalcent} in order to bound $\diam(\nc(G))$.

\begin{prop}[{\cite[Proposition 10]{nilppaper}}]
\label{prop:nomaxab}
Suppose that $G$ contains a normal non-abelian maximal subgroup $M$, with $Z(G) < Z(M)$. Then each maximal subgroup of $G$ is non-abelian.
\end{prop}

\begin{lem}
\label{lem:zneighbs}
Suppose that $G$ contains a normal, non-abelian maximal subgroup $M$ with $Z(G) < Z(M)$, and let $z \in Z(M) \setminus Z(G)$. Then $G \setminus M$ is the set of neighbours of $z$ in $\nc(G)$.
\end{lem}

\begin{proof}
As $M$ is non-abelian, $G/Z(M)$ is not cyclic. Thus Lemma~\ref{lem:noncycedge} yields the result.
\end{proof}

\begin{lem}
\label{lem:maxsubgendist}
Suppose that $G$ contains a normal, non-abelian maximal subgroup $M$ with $Z(G) < Z(M)$. In addition, let $x \in M \setminus Z(M)$ and $z \in Z(M) \setminus Z(G)$, and let $\mathcal{J}_M$ be the set of maximal subgroups of $G$ distinct from $M$. If $\langle I \cap M \mid I \in \mathcal{I} \rangle = M$ for some $\mathcal{I} \subseteq \mathcal{J}_M$, then there exists $I \in \mathcal{J}_M$ such that $x \notin C_M(I \cap M)$. More generally, if such $I$ exists, then $d(x,z) \le 3$.
\end{lem}

\begin{proof}
First, if ${\langle I \cap M \mid I \in \mathcal{I} \rangle} = M$ for some $\mathcal{I} \subseteq \mathcal{J}_M$, then $\bigcap_{I \in \mathcal{I}} C_M(I \cap M) = Z(M)$, and so there exists $I \in \mathcal{I}$ such that $x \notin C_M(I \cap M)$.

We now assume, more generally, that there exists $I \in \mathcal{J}_M$ such that $x \notin {C_M(I \cap M)}$. Then $C_{I \cap M}(x) < I \cap M$. Additionally, by Proposition~\ref{prop:nomaxab}, $I$ is non-abelian, and so there exists $s \in {I \setminus (Z(I) \cup M)}$. Applying Lemma~\ref{lem:centremaxmax}\ref{centremaxmax1} to $(s,I,M)$ now yields $C_{I \cap M}(s) < I \cap M$. Therefore, there exists an element $t \in I \cap M$ that centralises neither $x$ nor $s$. In addition, $s \sim z$ by Lemma~\ref{lem:zneighbs}. Thus $(x,t,s,z)$ is a path in $\nc(G)$, and $d(x,z) \le 3$.
\end{proof}

We can now bound distances in $\nc(G)$ between elements of $M \setminus Z(M)$ and elements of ${Z(M) \setminus Z(G)}$. Here, and in much of what follows, we will assume that $G$ is $2$-generated, as otherwise $\diam(\nc(G)) = 2$ by Proposition~\ref{prop:ncbasics}\ref{highgengp}.

\begin{prop}
\label{prop:xzdist}
Suppose that $G$ is $2$-generated and contains a normal, non-abelian maximal subgroup $M$ with $Z(G) < Z(M)$. In addition, let $x \in M \setminus Z(M)$ and $z \in Z(M) \setminus Z(G)$. Then the following statements hold.
\begin{enumerate}[label={(\roman*)},font=\upshape]
\item \label{xzdist1} $x$ and $z$ lie in distinct connected components of $\nc(G)$ if and only if $K \cap M = Z(M)$ for every maximal subgroup $K$ of $G$ distinct from $M$. Otherwise, $d(x,z) \le 4$.
\item \label{xzdist2} Suppose that $d(x,z) < \infty$. Then $d(x,z) = 4$ if and only if, for each maximal subgroup $K$ of $G$ distinct from $M$:
\begin{enumerate}[label={(\alph*)},font=\upshape]
\item $x \notin K$; and
\item $x \in C_M(K \cap M)$.
\end{enumerate}
\item \label{xzdist3} Suppose that \textnormal{(ii)(a)}--\textnormal{(b)} hold for each maximal subgroup $K$ of $G$ distinct from $M$. Then $K \cap M = \Phi(G)$ for all such $K$.
\item Suppose that $d(x,z) < \infty$, and that $G$ is finite. Then $d(x,z) \le 3$.
\end{enumerate}
\end{prop}

\begin{proof}
We first note that Proposition~\ref{prop:nomaxab} implies that $G$ contains no abelian maximal subgroups, while Lemma~\ref{lem:zneighbs} shows that $G \setminus M$ is the set of neighbours of $z$ in $\nc(G)$.

\medskip

\noindent (i) Suppose first that $K \cap M = Z(M)$ for every maximal subgroup $K$ of $G$ distinct from $M$, and let $y \in (G \setminus M)  \cup Z(M)$. Then $M$ is the unique maximal subgroup of $G$ containing $x$, and so if $\langle x, y \rangle < G$, then $y \in Z(M)$, and hence $[x,y] = 1$. Thus there is no edge in $\nc(G)$ between any element of $M \setminus Z(M)$ and any element of $(G \setminus M)  \cup Z(M) = G \setminus (M \setminus Z(M))$. In particular, the connected component of $\nc(G)$ containing $x$ consists only of elements of $M \setminus Z(M)$, and so this component does not contain $z \in Z(M)$.

Conversely, suppose that there exists a maximal subgroup $L$ of $G$ distinct from $M$ that satisfies $L \cap M \ne Z(M)$. We claim that $L \cap M \not\le Z(M)$. Indeed, either $Z(M) \not\le L \cap M$ or $L \cap M \not\le Z(M)$, and if the former holds, then applying Lemma~\ref{lem:threemaxint} to $(Z(M),M,L)$ yields $L \cap M \not\le Z(M)$. Additionally, as $L$ is non-abelian, there exists $r \in L \setminus (Z(L) \cup M)$. Applying Lemma~\ref{lem:centremaxmax}\ref{centremaxmax1} to $(r,L,M)$ yields $C_{L \cap M}(r) < L \cap M$, and so $Z(L) \cap M < L \cap M$. We also see, since $L \cap M \not\le Z(M)$, that $L \cap Z(M) < L \cap M$. Thus there exists an element $s \in (L \cap M) \setminus (Z(L) \cup Z(M))$, and an element $t \in L \setminus (C_L(s) \cup M)$. Proposition~\ref{prop:ncbasics}\ref{propernc} gives $d(x,s) \le 2$, and as $G \setminus M$ is the neighbourhood of $z$ in $\nc(G)$, we observe that $s \sim t \sim z$. Therefore, $d(x,z) \le d(x,s)+d(s,z) \le 4$.

\medskip

\noindent (ii) Assume first that (a) and (b) hold for each maximal subgroup $K$ of $G$ distinct from $M$. As $G \setminus M$ is the set of neighbours of $z$ in $\nc(G)$, it suffices by (i) to show that $d(x,t) \ge 3$ for all $t \in G \setminus M$. Suppose for a contradiction that $d(x,t) \le 2$ for some $t$. By (a), $\langle x, t \rangle = G$, and so $d(x,t) = 2$. Thus there exists $s \in M$ such that $x \sim s \sim t$, and so $\langle s, t \rangle$ lies in a maximal subgroup $R$ of $G$. However, $x$ centralises $s \in R \cap M$ by (b), a contradiction. Thus $d(x,z) = 4$.

Conversely, suppose that some maximal subgroup $K$ of $G$ distinct from $M$ fails to satisfy either (a) or (b). We will prove that $d(x,z) \le 3$. If $K$ does not satisfy (b), i.e., if $x \notin C_M(K \cap M)$, then this is an immediate consequence of Lemma~\ref{lem:maxsubgendist}, with $I = K$.

Assume therefore that $K$ does not satisfy (a), i.e., that $x \in K$. If $x \notin Z(K)$, then $C_K(x)$ and $C_K(z) = K \cap M$ are proper subgroups of $K$. Hence there exists $r \in K \setminus (M \cup C_K(x))$. As $G \setminus M$ is the neighbourhood of $z$ in $\nc(G)$, it follows that $x \sim r \sim z$ and $d(x,z) = 2$.

Suppose now that $x \in Z(K)$. Then $K = C_G(x)$, and since $x$ also lies in $M \setminus Z(G)$, applying Lemma~\ref{lem:nonabmax} to $(K,M)$ implies that $Z(M) \le K \cap M$, and it is now clear that $z \in K \setminus Z(K)$. If $K \trianglelefteq G$, then we obtain $d(x,z) \le 3$ by applying Lemma~\ref{lem:ncmaxnormcombined} to $(x,M,z,K)$.

If instead $K \not\trianglelefteq G$, then let $g \in G \setminus K$. If $x \in K^g$, then since $x \notin Z(K^g)$, applying the second last paragraph with $K^g$ replacing $K$ yields $d(x,z) = 2$. Otherwise, $x \notin K^g = C_G(x^g)$, and since $x^g \in K^g \cap M^g = K^g \cap M$, setting $I = K^g$ in Lemma~\ref{lem:maxsubgendist} gives $d(x,z) \le 3$.

\medskip

\noindent (iii) Let $\mathcal{J}_M$ be the set of maximal subgroups of $G$ distinct from $M$. Note that no $R \in \mathcal{J}_M$ is normal in $G$; otherwise, applying Lemma~\ref{lem:centremaxmax}\ref{centremaxmax1} to $(x,M,R)$ would imply that $x \notin C_{M}(R \cap M)$, contradicting (ii)(b). Suppose for a contradiction that $K \cap M \ne \Phi(G)$ for some $K \in \mathcal{J}_M$. We will show that there exists a subset $\mathcal{I}$ of $\mathcal{J}_M$ such that $\langle I \cap M \mid I \in \mathcal{I} \rangle = M$, and it will follow from Lemma~\ref{lem:maxsubgendist} that (ii)(b) does not hold for some $I \in \mathcal{J}_M$, a contradiction.

Suppose first that there exists $R \in \mathcal{J}_M$ such that $R \cap M \not\trianglelefteq G$. Since $R \cap M$ is a maximal subgroup of $R$, which is not normal in $G$, we see that $\langle R \cap M \rangle^G \not\le R$. However, $\langle R \cap M \rangle^G \le M$, since $M \trianglelefteq G$. If ${\langle R \cap M \rangle^G} \ne M$, then we can apply Lemma~\ref{lem:threemaxint} to the triple ${(\langle R \cap M \rangle^G,M,R)}$ of distinct subgroups, and this yields $R \cap M \not\le \langle R \cap M \rangle^G$, a contradiction. Therefore, $M = \langle R \cap M \rangle^G = {\langle R^g \cap M \mid g \in G \rangle}$, and so we can set $\mathcal{I} = \{R^g \mid g \in G\}$.

Assume finally that $R \cap M \trianglelefteq G$ for all $R \in \mathcal{J}_M$. Since $K \cap M \ne \Phi(G)$, Lemma~\ref{lem:twoconjclass} implies that there exists a maximal subgroup $L \in \mathcal{J}_M$ such that $L \cap M \ne K \cap M$. Either $K \cap M \not\le L$ or $L \cap M \not\le K$, and if the former holds, then since $K \cap M \trianglelefteq G$, applying Lemma~\ref{lem:threemaxint} to $(K,M,L)$ shows that the latter holds too. Thus, in general, $L \cap M \not\le K$. As $L \cap M \trianglelefteq G$, we obtain $(L \cap M)(K \cap M) = ((L \cap M)K)\cap M = G \cap M = M$. We can therefore set $\mathcal{I} = \{K,L\}$.

\medskip

\noindent (iv) By (i), $d(x,z) \le 4$, and there exists a maximal subgroup $K$ of $G$ distinct from $M$ with $K \cap M \ne Z(M)$. Suppose for a contradiction that $d(x,z) = 4$. Then (ii) shows that each maximal subgroup $L$ of $G$ distinct from $M$ satisfies $x \in C_M(L \cap M) \setminus L$. Additionally, (iii) implies that $\Phi(G) = {L \cap M}$ for each $L$, and so $x \in C_M(\Phi(G)) \setminus \Phi(G)$. Furthermore, Lemma~\ref{lem:twoconjclass} implies that the finite group $G$ contains exactly two conjugacy classes of maximal subgroups. Since $M$ is the unique normal maximal subgroup of $G$ (as in the proof of (iii)), and since $\Phi(G) = K \cap M$, the equivalent conditions of part (b) of Theorem~\ref{thm:solmaxclass}\ref{solmaxclass2} hold. In addition, $M$ is non-abelian and $\Phi(G) = K \cap M \ne Z(M)$, and so part (d) of that theorem shows that $C_M(\Phi(G)) \le \Phi(G)$, contradicting $x \in C_M(\Phi(G)) \setminus \Phi(G)$. Thus $d(x,z) \le 3$.
\end{proof}

\begin{ques}
\label{ques:xz}
Does there exist an infinite $2$-generated group $G$ that contains a normal, non-abelian maximal subgroup $M$ with $Z(G) < Z(M)$, and that satisfies the equivalent conditions of Proposition~\ref{prop:xzdist}\ref{xzdist2} for some $x \in M \setminus Z(M)$, so that $d(x,z) = 4$ for $z \in Z(M) \setminus Z(G)$?
\end{ques}

In the following result, we write $\hat H:=H/Z(M)$ when $H$ is a subgroup of $G$ containing $Z(M)$. Recall that an abstract group is primitive if it contains a core-free maximal subgroup, which is a point stabiliser for the corresponding primitive coset action.

\begin{prop}
\label{prop:gkzmprim}
Suppose that $G$ contains a normal, non-abelian maximal subgroup $M$, and a maximal subgroup $K$ with $K \cap M = Z(M)$. Then the following statements hold.
\begin{enumerate}[label={(\roman*)},font=\upshape]
\item \label{gkzmprim1} $\hat G$ is primitive, and is the semidirect product of its unique minimal normal subgroup $\hat M$ by its point stabiliser $\hat K$, which has prime order.
\item $\hat G$ is finite if and only if it is soluble. Hence $G$ is soluble if and only if $\hat G$ is finite.
\item \label{gkzmprim3} If $\hat G$ is infinite, then $\hat M$ is an infinite simple group, and $|\hat K|$ is odd.
\item \label{gkzmprim4} If $G$ contains a maximal subgroup $L$ with $L \ne M$ and $Z(M) < L \cap M$, then $\hat G$ is infinite.
\end{enumerate}
\end{prop}

\begin{proof}
As $Z(M)$ is not a maximal subgroup of $M$, we deduce that $K$ is not normal in $G$. However, $Z(M) = K \cap M$ is a maximal subgroup of $K$, and so $\mathrm{Core}_G(K) = Z(M)$. Thus $\hat{G}$ is primitive with a point stabiliser $\hat{K}$ of prime order, and $\hat G = \hat M \sd \hat K$. Furthermore, each nontrivial normal subgroup $\hat N$ of $\hat{G}$ contained in $\hat{M}$ intersects $\hat{K}$ trivially, and $\hat N \hat K = \hat G$. We therefore deduce that $\hat N = \hat M$, and so $\hat M$ is a minimal normal subgroup of $\hat G$.

Now, each finite group with an abelian maximal subgroup is soluble \cite{herstein}. Hence if $\hat{G}$ is finite, then it is soluble, as is $G$. Hence in this case $\hat M$ is the unique minimal normal subgroup of $\hat{G}$. If instead $\hat{G}$ is infinite, then since $\hat{K}$ is finite, \cite[Theorem 1.1]{smith} shows that $\hat{M}$ is a direct product of isomorphic infinite simple groups, and is again the unique minimal normal subgroup of $\hat G$. Hence $G$ is insoluble. Arguing as in the proof of \cite[Theorem 4.1]{herzog}, we deduce that $\hat M$ is simple and $|\hat K|$ is odd. Thus we have proved (i)--(iii).

Finally, suppose that $G$ has a maximal subgroup $L$ as in (iv). As $\hat M$ is maximal and the unique minimal normal subgroup of $\hat G$, the maximal subgroup $\hat L$ is core-free. Additionally, $\hat K \cap \hat M = 1 \ne \hat L \cap \hat M$, and thus the core-free maximal subgroup $\hat K$ is not conjugate to $\hat L$. Theorem~\ref{thm:finsolconjmax} therefore implies that $\hat G$ is either infinite or insoluble, and (iv) follows from (ii).
\end{proof}

\begin{lem}
\label{lem:rmaxdist}
Suppose that $G$ contains a normal, non-abelian maximal subgroup $M$. In addition, let $x \in M \setminus Z(M)$ and $y \in G \setminus M$. Finally, suppose that $R$ is a maximal subgroup of $G$ containing $y$, with $C_{R \cap M}(y) < R \cap M \ne Z(M)$. Then $d(x,y) \le 3$.
\end{lem}

\begin{proof}
Either $Z(M) \not\le R$ or $R \cap M \not\le Z(M)$, and if the former holds, then applying Lemma~\ref{lem:threemaxint} to $(Z(M),M,R)$ shows that the latter also holds. Thus, in general, $R \cap M \not\le Z(M)$, and so $R \cap Z(M) < R \cap M$. Since $C_{R \cap M}(y) < R \cap M$, there exists $h \in R \cap M$ with $h \notin C_{R \cap M}(y) \cup Z(M)$. We see that $h \sim y$, while $d(x,h) \le 2$ by Proposition~\ref{prop:ncbasics}\ref{propernc}. Therefore, $d(x,y) \le 3$.
\end{proof}

\begin{lem}
\label{lem:primgpint}
Suppose that $G$ contains a normal, non-abelian maximal subgroup $M$, with $Z(G) < Z(M)$. In addition, suppose that $G$ contains maximal subgroups $K$ and $L$, with $K \cap M = Z(M) \not\le L$ and $L \not\trianglelefteq G$. Then the following statements hold.
\begin{enumerate}[label={(\roman*)},font=\upshape]
\item $U:=L \cap M$ is a normal subgroup of $G$.
\item Let $s \in L \cap (K \setminus Z(M))$. Then $S:=\{[s,r] \mid r \in Z(M)\} \trianglelefteq G$, and $SU = M$.
\item \label{primgpint2} Each element of $K \setminus Z(M)$ lies in some $G$-conjugate of $L$.
\end{enumerate}
\end{lem}

\begin{proof}\leavevmode

\noindent (i) As $U \trianglelefteq L$ and $U \le M = C_M(Z(M))$, we obtain $U \trianglelefteq Z(M)L = G$.

\medskip

\noindent (ii) Let $x,y \in Z(M)$. As $Z(M) \trianglelefteq G$, it follows that $[s,y] \in Z(M)$. We calculate $[s,x][s,y] = [s,yx]$, and it follows that $S \le Z(M)$, and hence $S \trianglelefteq M$. Additionally, $x^s \in Z(M)$, and so $[s,x]^s = [s,x^s] \in S$. Thus $S^s = S$. As $s \in K \setminus Z(M)$, we see that $s \notin M$ and $S \trianglelefteq \langle M,s \rangle = G$.

Now, by (i), $U \trianglelefteq G$. For a subgroup $T$ of $G$, let $\overline T:=TU/U$, and for an element $g \in G$, let $\overline g:=Ug$. We observe that $U$ is maximal in $L$, and since $L \not\trianglelefteq G$, it follows that $\overline G$ is primitive with point stabiliser $\overline L$. Since $|\overline G|$ is not prime, it follows that $Z(\overline G) = 1$.

Let $r \in Z(M) \setminus L$. Then $C_{\overline G}(\overline r) = \overline M$. As $s \notin M$, we deduce that $\overline{[s,r]} = [\overline s, \overline r] \ne 1$, and thus $[s,r] \notin U = L \cap M$. Since $S \le M$, it follows that $S \not\le L$. Thus $SL = G$, and we conclude that $SU = S(L \cap M) = SL \cap M = G \cap M = M$.

\medskip

\noindent (iii) Let $k \in K \setminus Z(M)$. As $G = Z(M)L$, it follows that $k = zf$ for some ${z \in Z(M)}$ and some $f \in L \setminus Z(M)$. In fact, since $Z(M) \le K$, we see that $f = z^{-1}k \in K \setminus Z(M)$. Hence $f^{-1} \in L \cap (K \setminus Z(M))$.

Finally, let $S:={\{[f^{-1},r] \mid r \in Z(M)\}}$. As $U \trianglelefteq G$ by (i), we deduce that $Z(M)U/U \le M/U$, which is equal to $SU/U$ by (ii). Thus there exists $r \in Z(M)$ such that $U z^{-1} = U[f^{-1},r]$, and hence $[f^{-1},r]z = z[f^{-1},r] \in U$. As ${U = U^r \le L^r}$, it follows that $k = zf = z[f^{-1},r]f^r \in L^r$.
\end{proof}

We now bound distances in $\nc(G)$ between elements of $M \setminus Z(M)$ and elements of $G \setminus M$.

\begin{prop}
\label{prop:xydist}
Suppose that $G$ is $2$-generated and contains a normal, non-abelian maximal subgroup $M$, with $Z(G) < Z(M)$. In addition, let $x \in M \setminus Z(M)$ and $y \in G \setminus M$.
\begin{enumerate}[label={(\roman*)},font=\upshape]
\item \label{xydist1} $x$ and $y$ lie in distinct connected components of $\nc(G)$ if and only if $K \cap M = Z(M)$ for every maximal subgroup $K$ of $G$ distinct from $M$.
\item \label{xydist2} If $x$ and $y$ lie in the same connected component of $\nc(G)$, then $d(x,y) \le 4$.
\item \label{xydist3} If $d(x,y) = 4$, then $\Phi(G) = Z(M)$, and $G/Z(M)$ is primitive with unique minimal normal subgroup $M/Z(M)$, which is infinite and simple. Moreover, each maximal subgroup $K$ of $G$ containing $y$ satisfies $K \cap M = Z(M)$, and $K/Z(M)$ is a point stabiliser of $G/Z(M)=(M/Z(M)):(K/Z(M))$ of odd prime order. Additionally, $G$ contains a maximal subgroup $L$ such that $Z(M) < L \cap M$.
\end{enumerate}
\end{prop}

\begin{proof}
First, Proposition~\ref{prop:nomaxab} shows that $G$ contains no abelian maximal subgroups. Let $z \in Z(M) \setminus Z(G)$. Then $y \sim z$ by Lemma~\ref{lem:zneighbs}, and so $x$ and $y$ lie in the same connected component of $\nc(G)$ if and only if $x$ and $z$ lie in the same component. Thus (i) follows from Proposition~\ref{prop:xzdist}\ref{xzdist1}.

Assume now that $x$ and $y$ lie in the same connected component of $\nc(G)$.

\medskip

\noindent \textbf{Case (a)}: $\Phi(G) = Z(M)$. Let $K$ be a maximal subgroup of $G$ containing $y$, so that $K \ne M$, and suppose that $d(x,y) > 3$. Since $K \cap M$ contains $\Phi(G) = Z(M) > Z(G)$, we deduce from Lemma~\ref{lem:nonabmax}, applied to $(M,K)$, that $Z(K) \le K \cap M$. As $y \notin M$, it follows that $y \notin Z(K)$. Thus applying Lemma~\ref{lem:centremaxmax}\ref{centremaxmax1} to $(y,K,M)$ yields $C_{K \cap M}(y) < K \cap M$. Since $d(x,y) > 3$, Lemma~\ref{lem:rmaxdist} shows that $K \cap M = Z(M)$. On the other hand, as $x$ and $y$ lie in the same component of $\nc(G)$, we see from (i) that there exists a maximal subgroup $L$ of $G$ with $L \ne M$ and $L \cap M \ne Z(M)$. This means that $\Phi(G) = Z(M) < L \cap M$, and so Proposition~\ref{prop:xzdist} gives $d(x,z) \le 3$. As $y \sim z$, it follows that $d(x,y) = 4$. Furthermore, as $K \cap M = Z(M) < L \cap M$, Proposition~\ref{prop:gkzmprim}\ref{gkzmprim4} shows that $G/Z(M)$ is infinite, and hence the claims about $G/Z(M)$, $M/Z(M)$ and $K/Z(M)$ in (iii) follow from Proposition~\ref{prop:gkzmprim}\ref{gkzmprim1}--\ref{gkzmprim3}.

\medskip

\noindent \textbf{Case (b)}: $\Phi(G) \ne Z(M)$. To complete the proof of (ii) and (iii), it suffices to show that $d(x,y) \le 3$. Since $y \sim z$, some maximal subgroup $K$ of $G$ contains $y$ and $z$. Note that $z \in (K \cap M) \setminus C_G(y)$, and so $C_{K \cap M}(y) < K \cap M$. By Lemma~\ref{lem:rmaxdist}, we may assume that $K \cap M = Z(M)$, and so $\Phi(G) < Z(M)$. Hence some maximal subgroup $L$ satisfies $Z(M) \not\le L$.

We will show that $Z(L) \le Z(G)$. Observe that $G/Z(M) = Z(M)L/Z(M) \cong {L/(L \cap Z(M))}$. Since $G/Z(M)$ is primitive by Proposition~\ref{prop:gkzmprim} (and $|G/Z(M)|$ is not prime), $L/(L \cap Z(M))$ has trivial centre. Thus $Z(L) \le L \cap Z(M) \le M$. As $Z(M) \not\le L \cap M$, the contrapositive of Lemma~\ref{lem:nonabmax}, applied to $(L,M)$, shows that $Z(L) = Z(L) \cap M \le Z(G)$, as claimed.

We divide the remainder of Case (b) into three (not all mutually exclusive) sub-cases.

\medskip

\noindent \textbf{Case (b)($\boldsymbol{\alpha}$)}: $y \in L^g$ for some $g \in G$. Since $Z(L^g) \le Z(G)$, applying Lemma~\ref{lem:centremaxmax}\ref{centremaxmax1} to $(y,L^g,M)$ yields $C_{L^g \cap M}(y) < L^g \cap M \ne Z(M)$. Thus $d(x,y) \le 3$ by Lemma~\ref{lem:rmaxdist}.

\medskip

\noindent \textbf{Case (b)($\boldsymbol{\beta}$)}: $L \not\trianglelefteq G$. Since $y \in K \setminus Z(M)$, it follows from Lemma~\ref{lem:primgpint}\ref{primgpint2} that $y \in L^g$ for some $g \in G$. Thus by the previous sub-case, $d(x,y) \le 3$.

\medskip

\noindent \textbf{Case (b)($\boldsymbol{\gamma}$)}: $L \trianglelefteq G$ and $y \notin L$. Applying Lemma~\ref{lem:centremaxmax}\ref{centremaxmax1} to $(y,K,L)$ shows that $r \sim y$ for some $r \in K \cap L$, and that $C_{K \cap L}(y) < K \cap L$. If $x \in L$, then (since $Z(L) \le Z(G)$) Proposition~\ref{prop:ncbasics}\ref{propernc} yields $d(x,r) \le 2$, and so $d(x,y) \le 3$.

If instead $x \notin L$, then since $K \cap M = Z(M) \not\le L$, applying Lemma~\ref{lem:threemaxint} to $(K,M,L)$ shows that $L \cap M \not\le K$. Therefore, applying the same proposition to $(M,L,K)$ yields $K \cap L \not\le M$. Thus there exists $t \in (K \cap L) \setminus M$. In particular, $t \notin Z(G)$, and hence $t \notin Z(L)$. It follows from Lemma~\ref{lem:centremaxmax}\ref{centremaxmax2}, applied to the triple $(x,M,L)$ and the element $t$, that $x \sim s \sim t$ for some $s \in L \cap M$. As $s \sim t \in K \cap L$, we see that $C_{K \cap L}(s) < K \cap L$. Additionally, $C_{K \cap L}(y) < K \cap L$ by the previous paragraph. Hence there exists an element $f \in K \cap L$ that centralises neither $s$ nor $y$. Since $y \in K$, we see that $x \sim s \sim f \sim y$ and $d(x,y) \le 3$.
\end{proof}

\begin{ques}
\label{ques:xy}
Does there exist an infinite $2$-generated group $G$ that contains a normal, non-abelian maximal subgroup $M$ with $Z(G) < Z(M)$, that satisfies all necessary conditions given in Proposition~\ref{prop:xydist}\ref{xydist3}? If yes, is $d(x,y) = 4$ possible for $x \in M \setminus Z(M)$ and $y \in G \setminus M$?
\end{ques}

\begin{rem}
\label{rem:dist4}
Let $x \in M \setminus Z(M)$, $z \in Z(M) \setminus Z(G)$ and $y \in G \setminus M$. Propositions~\ref{prop:xzdist} and \ref{prop:xydist} show that if $d(x,z) = 4$, then $K \cap M = \Phi(G) \ne Z(M)$ for each maximal subgroup $K$ of $G$ distinct from $M$, while if $d(x,y) = 4$, then there exists a maximal subgroup $L \ne M$ such that $Z(M) = \Phi(G) < L \cap M$. Hence $d(x,z)$ and $d(x,y)$ cannot both be equal to $4$.
\end{rem}

Our next result specifies exactly when $\nc(G)$ is connected, assuming that $G$ contains a maximal subgroup $M$ as above. In the next section, we will consider in more detail the diameters of the connected components of this graph, and discuss several concrete examples.

\begin{lem}
\label{lem:maxnormconnected}
Suppose that $G$ is $2$-generated and contains a normal, non-abelian maximal subgroup $M$, with $Z(G) < Z(M)$. Then $\nc(G)$ is not connected if and only if $K \cap M = Z(M)$ for every maximal subgroup $K$ of $G$ distinct from $M$, in which case Proposition~\ref{prop:gkzmprim} applies to $G$ for any choice of $K$. In particular, if $G$ is finite, then $\nc(G)$ is not connected if and only if $G$ satisfies Assumption~\ref{assump:22g}.
\end{lem}

\begin{proof}
Suppose first that $K \cap M = Z(M)$ for every maximal subgroup $K$ of $G$ distinct from $M$. Then Propositions~\ref{prop:xzdist} and~\ref{prop:xydist} show that there is no path in $\nc(G)$ between any element of $M \setminus Z(M)$ and any element of $(G \setminus M) \cup (Z(M) \setminus Z(G))$. Hence $\nc(G)$ is not connected. In addition, Proposition~\ref{prop:gkzmprim} applies to $G$ for any choice of $K$, and so $\mathrm{Core}_G(K) = Z(M)$. Thus $M$ is the unique normal maximal subgroup of $G$. It follows from Lemma~\ref{lem:twoconjclass} that if $G$ is finite, then it is soluble and contains exactly two conjugacy classes of maximal subgroups, and $Z(M) = \Phi(G)$ is the intersection of each pair of distinct maximal subgroups. Hence in this case $G$ satisfies all conditions of Theorem~\ref{thm:solmaxclass}\ref{solmaxclass1}. In particular, $G$ has exactly two conjugacy classes of maximal subgroups, and a unique non-cyclic Sylow subgroup $P$. Furthermore, Theorem~\ref{thm:solmaxclass}\ref{solmaxclass2} shows that $G$ has a nontrivial cyclic Sylow subgroup whose maximal subgroup is normal in $G$, and that $\Phi(P) = Z(P)$. As $Z(G) < Z(M) = \Phi(G)$, we also observe from this theorem that $\Phi(P) \not\le Z(G)$. Thus $G$ satisfies Assumption~\ref{assump:22g}.

If instead $G$ has a maximal subgroup $L \ne M$ with $L \cap M \ne Z(M)$, then Propositions~\ref{prop:xzdist} and~\ref{prop:xydist} imply that $\nc(G)$ is connected. Suppose that $G$ is finite in this case. To show that Assumption~\ref{assump:22g} does not hold for $G$, we may assume that $G$ has exactly two conjugacy classes of maximal subgroups, and a nontrivial cyclic Sylow subgroup whose maximal subgroup is normal in $G$. Then Theorem~\ref{thm:solmaxclass}\ref{solmaxclass2}\ref{solmaxclass2c}--\ref{solmaxclass2d} implies that $\Phi(G) = L \cap M \ne Z(M)$, and hence that the unique non-cyclic Sylow subgroup $P$ of $G$ does not satisfy $\Phi(P) = Z(P)$. As $\Phi(S) < S = Z(S)$ for each nontrivial Sylow subgroup $S \ne P$, Assumption~\ref{assump:22g} is not satisfied by $G$.
\end{proof}

We note that Proposition~\ref{prop:gkzmprim} and Lemma~\ref{lem:maxnormconnected} show that if $G$ is infinite and $\nc(G)$ is not connected, then $G/Z(M)$ is primitive with a unique minimal normal subgroup, which is infinite and simple, and each point stabiliser of $G/Z(M)$ has odd prime order.

\section{Non-central by non-cyclic groups}
\label{sec:normalcent}

In this section, we will determine upper bounds (or exact values in some cases) for the diameters of the connected components of $\nc(G)$ whenever $G$ satisfies the following assumption.

\begin{assump}
\label{assump:nc}
Assume that $G$ contains a normal subgroup $N$, such that $G/N$ is not cyclic and $N \not\le Z(G)$. Additionally, let $C:=C_G(N)$.
\end{assump}

Note that this assumption holds whenever $G/Z(G)$ has a proper non-cyclic quotient. Additionally, $C$ and $Z(C)$ are normal subgroups of $G$, and $Z(G) \le C < G$.

Throughout this section, we will implicitly use Proposition~\ref{prop:ncbasics}\ref{nodiamonenc}, which states that each nontrivial connected component of $\nc(G)$ has diameter at least $2$.

\begin{lem}
\label{lem:gnbasefacts}
Let $G$, $N$ and $C$ be as in Assumption~\ref{assump:nc}.
\begin{enumerate}[label={(\roman*)},font=\upshape]
\item \label{gnbasefacts1} Let $h,h' \in G \setminus C$. Then there exists $n \in N \setminus Z(G)$ such that $h \sim n \sim h'$.
\item \label{gnbasefacts2} Let $c \in C \setminus Z(G)$ and $g \in G \setminus Z(G)$. If $d(c,g) > 2$, then either $G/\langle c \rangle^G$ and $G/\langle g \rangle^G$ are both cyclic, or one of these quotients is cyclic and $[c,g] = 1$.
\end{enumerate}
\end{lem}

\begin{proof} To prove (i), note that since $h,h' \notin C$, each of $C_N(h)$ and $C_N(h')$ is a proper subgroup of $N$. Thus there exists $n \in N \setminus (C_N(h) \cup C_N(h'))$, and Lemma~\ref{lem:noncycedge} yields $h \sim n \sim h'$.

Next, we prove the contrapositive of (ii). If $[c,g] \ne 1$ and either $G/\langle c \rangle^G$ or $G/\langle g \rangle^G$ is not cyclic, then Lemma~\ref{lem:noncycedge} yields $d(c,g) = 1$. Suppose therefore that $[c,g] = 1$, with $G/\langle c \rangle^G$ and $G/\langle g \rangle^G$ both non-cyclic. Since the non-commuting graph of $G$ has diameter $2$ by Proposition~\ref{prop:noncomdiam}, there exists $k \in G \setminus Z(G)$ such that $(c,k,g)$ is a path in that graph. By Lemma~\ref{lem:noncycedge}, this is also a path in $\nc(G)$, and hence $d(c,g) \le 2$.
\end{proof}

We now split the investigation of the structure of $\nc(G)$ into three cases: $G/C$ non-cyclic; $G/C$ cyclic and $C$ abelian; and $G/C$ cyclic and $C$ non-abelian. In the second and third cases, we will see that more can be said if we know whether or not $C$ is a maximal subgroup of $G$.

\begin{lem}
\label{lem:gcnoncyc}
Let $G$, $N$ and $C$ be as in Assumption~\ref{assump:nc}, and suppose that $G/C$ is not cyclic. Then $\nc(G)$ is connected with diameter $2$ or $3$. Moreover, if $d(x,y)=3$ for $x,y \in G \setminus Z(G)$, then one of these elements lies in $C$, the other lies in $G \setminus (N \cup C)$, and $[x,y] = 1$. Hence $\diam(\nc(G)) = 2$ if $C_G(x) \subseteq N \cup C$ for all $x \in C \setminus Z(G)$, and in particular if $C = Z(G)$.
\end{lem}

\begin{proof}
By Lemma~\ref{lem:gnbasefacts}\ref{gnbasefacts1}, any two elements of $G \setminus C$ are joined in $\nc(G)$ by a path of length at most two. Thus it suffices to consider distances in $\nc(G)$ involving elements of $C \setminus Z(G)$.

Suppose that $x \in C \setminus Z(G)$ and $y \in G \setminus Z(G)$ satisfy $d(x,y) > 2$. As $G/N$ and $G/C$ are not cyclic, neither is $G/\langle r \rangle^G$ for any $r \in N \cup C$. In particular, $G/\langle x \rangle^G$ is not cyclic. Therefore, Lemma~\ref{lem:gnbasefacts}\ref{gnbasefacts2} implies that $y \in G \setminus (N \cup C)$ and $[x,y] = 1$.

Now, Lemma~\ref{lem:gnbasefacts}\ref{gnbasefacts1} shows that $n \sim y$ for some $n \in N \setminus Z(G)$. By the previous paragraph, $d(x,n) \le 2$, and so $d(x,y) \le d(x,n) + d(n,y) \le 3$.
\end{proof}

Using Magma, we see that the groups $S_4$ and $C_2^2 \times S_3$ satisfy the hypotheses of Lemma~\ref{lem:gcnoncyc}, and have non-commuting, non-generating graphs of diameter $3$ and $2$, respectively. In fact, in the latter case, $C = Z(G)$. On the other hand, if $G = S_3 \times S_3$, then $G$ satisfies the hypotheses of Lemma~\ref{lem:gcnoncyc} and $\diam(\nc(G)) = 2$, even though $C_G(x) \not\subseteq N \cup C$ for some $x \in C \setminus Z(G)$.

\begin{eg}
\label{eg:thompf}
Consider the infinite, $2$-generated Thompson's group $F$. The derived subgroup $F'$ of $F$ is infinite and simple, $F/F' \cong \mathbb{Z}^2$, and every proper quotient of $F$ is abelian~\cite[\S1.4]{belkphd}. Hence $F'$ is the unique minimal normal subgroup of $F$, and it follows that $C_F(F') = 1$. As $F/F'$ is not cyclic, we can apply Lemma~\ref{lem:gcnoncyc} with $G = F$ and $N = F'$ to deduce that $\diam(\nc(F)) = 2$.
\end{eg}

Next, we prove useful properties of subgroups of $G$ containing $C$, when $G/C$ is cyclic.

\begin{lem}
\label{lem:ninzc}
Let $G$, $N$ and $C$ be as in Assumption~\ref{assump:nc}, and suppose that $G/C$ is cyclic. Additionally, let $H$ be a subgroup of $G$ properly containing $C$. Then $H$ is non-abelian, $Z(H) < Z(C)$, and $H \trianglelefteq G$. In particular, $Z(G) < Z(C)$.

\end{lem}

\begin{proof}\leavevmode
Since $G/C$ is cyclic, so is its subgroup $NC/C \cong N/(N \cap C) = N/Z(N)$. Thus $N$ is abelian, and it follows that $N \le Z(C)$. In particular, $H$ contains $N$. Hence each of $N$ and $C$ is centralised by $Z(H)$, and so $Z(H) \le Z(C)$. However, $H$ does not centralise $N$. Thus $N \not\le Z(H)$, and it follows that $H$ is non-abelian and $Z(H) < Z(C)$. Additionally, $H/C$ is a normal subgroup of the cyclic group $G/C$, and thus $H \trianglelefteq G$.
\end{proof}

Our next proposition explores the case where $G/C$ is cyclic and $C$ is abelian. Recall that $\nd(G)$ is the subgraph of $\nc(G)$ induced by its non-isolated vertices. As in much of the previous section, we assume that $G$ is $2$-generated; otherwise, $\diam(\nc(G)) = 2$ by Proposition~\ref{prop:ncbasics}\ref{highgengp}.

\begin{lem}
\label{lem:gccycab}
Let $G$, $N$ and $C$ be as in Assumption~\ref{assump:nc}. Suppose also that $G$ is $2$-generated, $G/C$ is cyclic, and $C$ is abelian, so that $G$ is soluble. Then the following statements hold.
\begin{enumerate}[label={(\roman*)},font=\upshape]
\item Each isolated vertex of $\nc(G)$ lies in $C \setminus N$.
\item \label{gccycab3} Suppose that $C$ is maximal in $G$. Then $\nd(G)$ is connected with diameter $2$.
\item \label{gccycab4} Suppose that $C$ is not maximal in $G$, and let $M$ be a maximal subgroup of $G$ containing $C$. Then Table~\ref{table:gcabnonmax} lists upper bounds for distances between vertices of $\nc(G)$, depending on the subsets of $G \setminus Z(G)$ that contain them. In particular, $\diam(\nc(G)) \le 3$.
\end{enumerate}
\end{lem}

\begin{table}[ht]
\centering
\renewcommand{\arraystretch}{1.1}
\caption{Upper bounds for distances between vertices $x \in A$ and $y \in B$ of $\nc(G)$, with $A,B \subseteq G \setminus Z(G)$, and $C$ and $M$ as in Lemma~\ref{lem:gccycab}\ref{gccycab4}.}
\label{table:gcabnonmax}
\begin{tabular}{|c|*{4}{c|}}
\hline
\diagbox[width=7em]{{\parbox{2.2cm}{\centering \vspace{0.0cm} $A$}}}{{\parbox{0.3cm}{\centering \vspace{0.2cm} \hspace{-1.2cm} $B$}}}
& $Z(M) \setminus Z(G)$ & $C \setminus Z(M)$ & $M \setminus C$ & $G \setminus M$ \\
\hline
\multirow{2}{*}{$G \setminus M$} & \multirow{2}{*}{$1$} & $3$ & \multirow{2}{*}{$2$} & \multirow{2}{*}{$2$} \\
 &  &  $2$, if $[x,y] \ne 1$ & & \\
\hline
$M \setminus C$ & $3$ & $2$ & $2$ \\
\cline{1-4}
$C \setminus Z(M)$ & $3$ & $2$  \\
\cline{1-3}
$Z(M) \setminus Z(G)$ & $2$ \\
\cline{1-2}
\end{tabular}
\end{table}

\begin{proof}\leavevmode

\noindent (i) By Lemma~\ref{lem:gnbasefacts}\ref{gnbasefacts1}, any two elements of $G \setminus C$ have distance at most two in $\nc(G)$. Additionally, as $G/N$ is not cyclic, and as the non-commuting graph of $G$ is connected by Proposition~\ref{prop:noncomdiam}, it follows from Lemma~\ref{lem:noncycedge} that each isolated vertex lies in $C \setminus N$.

\medskip

\noindent (ii) By Lemma~\ref{lem:gnbasefacts}\ref{gnbasefacts1}, it suffices to show that $d(x,y) \le 2$ whenever $x \in C \setminus Z(G)$ and $y \in G \setminus Z(G)$ are distinct non-isolated vertices. If $G/\langle x \rangle^G$ is not cyclic, then Lemma~\ref{lem:noncycedge} shows that $G \setminus C_G(x) = G \setminus C$ is the neighbourhood of $x$ in $\nc(G)$. In particular, if $y \in G \setminus C$, then $d(x,y) = 1$. If instead the non-isolated vertex $y$ lies in the abelian group $C$, then $k \sim y$ for some $k \in G \setminus C$. Hence $x \sim k \sim y$ and $d(x,y) = 2$.

Suppose now that $G/\langle x \rangle^G$ is cyclic. By Proposition~\ref{prop:ncbasics}\ref{isolvertnc}, there exist maximal subgroups $L$ and $K$ of $G$ with $x \in L \setminus Z(L)$ and $y \in K \setminus Z(K)$. Then $x \in C \cap L$, and applying Lemma~\ref{lem:nonabmax} to $(C,L)$ gives $Z(L) \le Z(G)$. Note also that $G = CL = C_G(x)L$, and so $\langle x \rangle^G = \langle x \rangle^L \le L$. Thus $L/\langle x \rangle^G \trianglelefteq G/\langle x \rangle^G$, and hence $L \trianglelefteq G$. This implies that $C_L(x) = C \cap L \trianglelefteq G$. Additionally, $y \notin Z(L) \le Z(G)$, and $x$ is centralised by $C$, and hence not by $K$. We therefore obtain $d(x,y) \le 2$ by applying Lemma~\ref{lem:ncmaxnormcombined} to $(x,L,y,K)$.

\medskip

\noindent (iii) Since $Z(G) \le C < M$, it follows that $Z(G) \le Z(M)$. Additionally, Lemma~\ref{lem:ninzc} shows that $M$ is non-abelian and normal in $G$, with $Z(M) < Z(C) = C$. We observe from Proposition~\ref{prop:ncbasics}\ref{propernc} and Lemma~\ref{lem:gnbasefacts}\ref{gnbasefacts1} that any two vertices of $\nc(G)$ in $M \setminus Z(M) = (M \setminus C) \cup (C \setminus Z(M))$ have distance at most two, as do any two vertices in $G \setminus C = (G \setminus M) \cup (M \setminus C)$. This yields the $(1,3)$, $(1,4)$, $(2,2)$, $(2,3)$ and $(3,2)$ entries of Table~\ref{table:gcabnonmax}.

Now, suppose that $Z(G) < Z(M)$, and let $z \in Z(M) \setminus Z(G)$. Since $M$ is a non-abelian, normal subgroup of $G$ and $\langle z \rangle^G \le Z(M)$, the quotient $G/\langle z \rangle^G$ is not cyclic. Additionally, $C_G(z) = M$, and so Lemma~\ref{lem:noncycedge} gives $z \sim r$ for each $r \in G \setminus M$, hence the $(1,1)$ entry of Table~\ref{table:gcabnonmax}. Thus if $z' \in Z(M) \setminus Z(G)$ is not equal to $z$, then $z \sim r \sim z'$ and $d(z,z') = 2$, yielding the $(4,1)$ entry of the table. Moreover, if $m \in M \setminus C$, then $d(r,m) \le 2$ by the (1,3) entry of the table, and so $d(z,m) \le d(z,r)+d(r,m) \le 3$. This gives the $(2,1)$ entry of the table.

It remains to determine upper bounds for $d(c,g)$, where $c \in {C \setminus Z(M)}$ and $g \in G \setminus M$, and for $d(c,z)$ when the element $z$ exists. As $g$ does not lie in $C$, it is a non-isolated vertex by (i). It follows from Proposition~\ref{prop:ncbasics}\ref{isolvertnc} that $g \in K \setminus Z(K)$ for some maximal subgroup $K$ of $G$. Additionally, the abelian group $C$ lies in $C_G(c)$, and the cyclic group $G/C$ normalises $C_G(c)/C$. Thus $C_G(c) \trianglelefteq G$. Since $M \trianglelefteq G$, it follows that $C_M(c) = C_G(c) \cap M \trianglelefteq G$. Therefore, applying Lemma~\ref{lem:ncmaxnormcombined} to $(c,M,g,K)$ gives $d(c,g) \le 3$. Moreover, since $g \notin Z(M)$, that lemma shows that if $d(c,g) = 3$, then $c \in Z(K)$, and in particular, $[c,g] = 1$. Thus we obtain the $(1,2)$ entry of Table~\ref{table:gcabnonmax}.

Finally, since $C_G(c) < G$ and $M < G$, there exists $h \in {G \setminus (M \cup C_G(c))}$. The $(1,1)$ and $(1,2)$ entries of Table~\ref{table:gcabnonmax} yield $h \sim z$ and $d(c,h) \le 2$. Hence $d(c,z) \le d(c,h)+d(h,z) \le 3$. This gives the $(3,1)$ entry of the table, completing the proof.
\end{proof}

Using Magma, we observe that if $G$ is equal to $C_3 \sd S_3$ or the dihedral group $D_{12}$ of order $12$, then Lemma~\ref{lem:gccycab}\ref{gccycab3} applies, and $\nc(G)$ is connected only in the former case (in both cases, $\nd(G)$ is connected with diameter $2$). If instead $G$ is equal to $C_2 \times \mathrm{AGL}(1,5)$ or $C_3 \sd \mathrm{AGL}(1,5)$, then Lemma~\ref{lem:gccycab}\ref{gccycab4} applies, and $\nc(G)$ has diameter $2$ or $3$, respectively.

The following result is a more detailed version of Lemma~\ref{lem:maxnormconnected}, with a weaker hypothesis.

\begin{lem}
\label{lem:gccycnonab}
Let $G$ and $C$ be as in Assumption~\ref{assump:nc}. Suppose also that $G$ is $2$-generated, $G/C$ is cyclic, and $C$ is non-abelian. Then the following statements hold.
\begin{enumerate}[label={(\roman*)},font=\upshape]

\item \label{gccycnonab1} $\nc(G)$ is not connected if and only if $C$ is maximal in $G$ and $K \cap C = Z(C)$ for every maximal subgroup $K$ of $G$ distinct from $C$. In this case, $\nc(G)$ is the union of two connected components of diameter $2$, and one component consists of the elements of ${C \setminus Z(C)}$. In particular, if $G$ is finite, then $\nc(G)$ is not connected if and only if $G$ satisfies Assumption~\ref{assump:22g}.

\item \label{gccycnonab2} Suppose that $C$ is not maximal in $G$, and let $M$ be a maximal subgroup of $G$ containing $C$. Then Table~\ref{table:gcnonabnonmax} lists upper bounds for distances between vertices of $\nc(G)$, depending on the subsets of $G \setminus Z(G)$ that contain them. In particular, $\diam(\nc(G)) \le 4$, and $\diam(\nc(G)) \le 3$ if $G$ is finite or if $Z(M) = Z(G)$.

\item \label{gccycnonab3} Suppose that $C$ is maximal in $G$, and that $\nc(G)$ is connected. Then Table~\ref{table:gcnonabmax} lists upper bounds for distances between vertices of $\nc(G)$, depending on the subsets of $G \setminus Z(G)$ that contain them. In particular, $\diam(\nc(G)) \le 4$, and $\diam(\nc(G)) \le 3$ if $G$ is finite.
\end{enumerate}
\end{lem}

\begin{table}[ht]
\centering
\renewcommand{\arraystretch}{1.1}
\caption{Upper bounds for distances between vertices $x \in A$ and $y \in B$ of $\nc(G)$, with $A,B \subseteq G \setminus Z(G)$, and $C$ and $M$ as in Lemma~\ref{lem:gccycnonab}\ref{gccycnonab2}. Additionally, $\mathcal{M}$ denotes the family of groups for which any two distinct maximal subgroups intersect in $\Phi(G)$.}
\label{table:gcnonabnonmax}
\begin{tabular}{|c||*{5}{c|}}
\hline
\diagbox[width=7em]{{\parbox{2.2cm}{\centering \vspace{0.0cm} $A$}}}{{\parbox{0.3cm}{\centering \vspace{0.2cm} \hspace{-1.2cm} $B$}}}
& $Z(M) \setminus Z(G)$ & $Z(C) \setminus Z(M)$ & $C \setminus Z(C)$ & $M \setminus C$ & $G \setminus M$ \\
\hline
\hline
$G \setminus M$ & $1$ & $3$ & $3$ & $2$ & $2$  \\
\hline
$M \setminus C$ & $3$ & $2$ & $2$  & $2$  \\
\cline{1-5}
\multirow{3}{*}{$C \setminus Z(C)$} & $4$ & \multirow{3}{*}{$2$}  & \multirow{3}{*}{$2$} \\
& $3$, if $|G| < \infty$ & & \\
& $3$, if $G \notin \mathcal{M}$ & & \\
\cline{1-4}
$Z(C) \setminus Z(M)$ & $2$ & $2$ \\
\cline{1-3}
$Z(M) \setminus Z(G)$ & $2$ \\
\cline{1-2}
\end{tabular}
\end{table}

\begin{table}[ht]
\centering
\renewcommand{\arraystretch}{1.1}
\caption{Upper bounds for distances between vertices $x \in A$ and $y \in B$ of $\nc(G)$, with $A,B \subseteq G \setminus Z(G)$, and $C$ as in Lemma~\ref{lem:gccycnonab}\ref{gccycnonab3}. Additionally, $\mathcal{M}$ denotes the family of groups for which any two distinct maximal subgroups intersect in $\Phi(G)$.}
\label{table:gcnonabmax}
\begin{tabular}{|c||*{3}{c|}}
\hline
\diagbox[width=7em]{{\parbox{2.2cm}{\centering \vspace{0.0cm} $A$}}}{{\parbox{0.3cm}{\centering \vspace{0.2cm} \hspace{-1.2cm} $B$}}}
& $Z(C) \setminus Z(G)$ & $C \setminus Z(C)$ & $G \setminus C$ \\
\hline
\hline
\multirow{3}{*}{$G \setminus C$} & \multirow{3}{*}{$1$} & $4$ & \multirow{3}{*}{$2$} \\
 & & $3$, if $|G/Z(C)| < \infty$ & \\
 & & $3$, if $G \in \mathcal{M}$ & \\
\hline
\multirow{3}{*}{$C \setminus Z(C)$} & $4$ & \multirow{3}{*}{$2$} \\
& $3$, if $|G| < \infty$ & \\
& $3$, if $G \notin \mathcal{M}$ & \\
\cline{1-3}
$Z(C) \setminus Z(G)$ & $2$ \\
\cline{1-2}
\end{tabular}
\end{table}

\begin{proof}
Proposition~\ref{prop:ncbasics}\ref{propernc} and Lemma~\ref{lem:gnbasefacts}\ref{gnbasefacts1} show that any two vertices in $C \setminus Z(C)$ are joined by a path of length at most two, as are any two vertices in $G \setminus C$, and any two vertices in $M \setminus Z(M)$ when $M$ is as in (ii). We therefore obtain the $(1,4)$, $(1,5)$, $(2,2)$, $(2,3)$, $(2,4)$, $(3,2)$, $(3,3)$ and $(4,2)$ entries of Table~\ref{table:gcnonabnonmax}, and the $(1,3)$ and $(2,2)$ entries of Table~\ref{table:gcnonabmax}. Additionally, Lemma~\ref{lem:ninzc} implies that $Z(G) < Z(C)$. Let $z \in Z(C) \setminus Z(G)$. Since $C$ is non-abelian and normal in $G$ and $\langle z \rangle^G \le Z(C)$, the quotient $G/\langle z \rangle^G$ is not cyclic. We split the remainder of the proof into two cases, depending on whether $C$ is maximal in $G$.

\medskip

\noindent \textbf{Case (a)}: $C$ is maximal in $G$. Here, $C = C_G(z)$. Since $G/\langle z \rangle^G$ is not cyclic, it follows from Lemma~\ref{lem:noncycedge} that $z \sim k$ for all $k \in G \setminus C$. Thus we obtain the $(1,1)$ entry of Table~\ref{table:gcnonabmax}. If $z'$ is another element of ${Z(C) \setminus Z(G)}$, then $z \sim k \sim z'$, yielding the $(3,1)$ entry of Table~\ref{table:gcnonabmax}. Furthermore, 
since ${Z(G) < Z(C)}$, Lemma~\ref{lem:maxnormconnected} shows that if $G$ is finite, then $\nc(G)$ is not connected if and only $G$ satisfies Assumption~\ref{assump:22g}, and in general, $\nc(G)$ is not connected if and only if $K \cap C = Z(C)$ for every maximal subgroup $K \ne C$ of $G$.

Suppose first that $\nc(G)$ is connected, and let $x \in C \setminus Z(C)$ and $y \in G \setminus C$. Then Propositions~\ref{prop:xzdist} shows that $d(x,z) \le 4$, and that if $d(x,z) = 4$, then $|G| = \infty$ and $K \cap C = \Phi(G)$ for each maximal subgroup $K$ of $G$ distinct from $C$. It follows from Lemma~\ref{lem:twoconjclass} that if $d(x,z) = 4$, then any two distinct maximal subgroups of $G$ intersect in $\Phi(G)$, and we obtain the $(2,1)$ entry of Table~\ref{table:gcnonabmax}. Additionally, Proposition~\ref{prop:xydist} shows that $d(x,y) \le 4$, and that if $d(x,y) = 4$, then $|G/Z(C)| = \infty$ and $G$ contains maximal subgroups $K$ and $L \ne C$ such that $K \cap C = Z(C) \ne L \cap C$. This yields the $(1,2)$ entry of Table~\ref{table:gcnonabmax}. We have therefore proved (iii).

Next, suppose that $\nc(G)$ is not connected. We have shown that if $g,h \in C \setminus Z(C)$ or $g,h \in (G \setminus C) \cup (Z(C) \setminus Z(G))$, then $d(g,h) \le 2$. Hence the components of $\nc(G)$ and their diameters are as in (i). To complete the proof of (i), it remains to show that if $C$ is not maximal in $G$, then $\nc(G)$ is connected, and if $G$ is also finite, then it does not satisfy Assumption~\ref{assump:22g}.

\medskip

\noindent \textbf{Case (b)}: $C$ is not maximal in $G$. Let $M$ be a maximal subgroup of $G$ containing $C$. Then $M$ is non-abelian, and Lemma~\ref{lem:ninzc} gives $M \trianglelefteq G$ and $Z(M) < Z(C)$. As $Z(G) < C$, it also follows that $Z(G) \le Z(M)$. 
Let $z' \in Z(C) \setminus (Z(G) \cup \{z\})$. Since $G/\langle z \rangle^G$ and $G/\langle z' \rangle^G$ are not cyclic, we can apply Lemma~\ref{lem:gnbasefacts}\ref{gnbasefacts2} (and the fact that $[z,z'] = 1$) to obtain $d(z,z') = 2$. This gives the $(4,1)$ and $(5,1)$ entries of Table~\ref{table:gcnonabnonmax}. Note also that as $z \in Z(C) \setminus Z(G)$, there exists $h \in G \setminus C$ with $[z,h] \ne 1$, and Lemma~\ref{lem:noncycedge} gives $z \sim h$. Letting $g \in G \setminus C$, the known entries of Table~\ref{table:gcnonabnonmax} show that $d(h,g) \le 2$, and so $d(z,g) \le 3$. This yields the $(1,2)$ and $(2,1)$ entries of the table.

Now, let $x \in C \setminus Z(C)$. Then $x \notin Z(M)$, and so $C_M(x) < M$. Thus there exists $k \in {M \setminus (C \cup C_M(x))}$, and we observe that $x \sim k$. The $(1,4)$ entry of Table~\ref{table:gcnonabnonmax} gives $d(k,g) \le 2$ for each $g \in G \setminus M$, and hence $d(x,g) \le 3$, yielding the $(1,3)$ entry of the table.

Next, we will consider the remaining entries in the first column of Table~\ref{table:gcnonabnonmax}, which apply only when $Z(G) < Z(M)$. Assume that $z \in Z(M) \setminus Z(G)$. Then the $(2,1)$ entry of Table~\ref{table:gcnonabnonmax} shows that $d(z,m) < \infty$ for each $m \in M \setminus C$. The $(3,1)$ entry of the table therefore follows from Proposition~\ref{prop:xzdist} and Lemma~\ref{lem:twoconjclass}. In addition, $M = C_G(z)$, and so Lemma~\ref{lem:noncycedge} gives the $(1,1)$ entry of the table. This completes the proof of (ii).

We have shown that $\nc(G)$ is connected, which partially proves (i). To complete the proof, suppose for a contradiction that $G$ is finite and satisfies Assumption~\ref{assump:22g}. Then the unique non-cyclic Sylow subgroup $G$ of $P$ satisfies $Z(P) \not\le Z(G)$; the unique maximal subgroup of each nontrivial cyclic Sylow subgroup of $G$ is normal in $G$; and $M$ is precisely the maximal subgroup $M$ specified in Theorem~\ref{thm:solmaxclass}\ref{solmaxclass2} (by part (a) of that theorem). Part (c) of that theorem therefore implies that $Z(M)$ contains $Z(P) \not\le Z(G)$, and so $Z(G) < Z(M)$. Hence Lemma~\ref{lem:maxnormconnected} applies, and shows that $G$ does not in fact satisfy Assumption~\ref{assump:22g}. Thus (i) holds.
\end{proof}

Notice from Propositions~\ref{prop:xzdist} and~\ref{prop:xydist} that the conditions $G \notin \mathcal{M}$ and $G \in \mathcal{M}$ in the $(2,1)$ and $(1,2)$ entries of Table~\ref{table:gcnonabmax}, respectively, are stronger than those necessary to ensure that the specified distances cannot be equal to $4$ (and similarly for the $(3,1)$ entry of Table~\ref{table:gcnonabnonmax}). However, the chosen conditions highlight the fact that there is no group for which these two entries of Table~\ref{table:gcnonabmax} are simultaneously equal to $4$, as discussed in Remark~\ref{rem:dist4}.

We compute via Magma that $S_3 \times \mathrm{AGL}(1,5)$ and the group numbered $(192,30)$ in the Small Groups Library \cite{smallgroups} satisfy the hypotheses of Lemma~\ref{lem:gccycnonab}\ref{gccycnonab2}, and have non-commuting, non-generating graphs of diameter $2$ and $3$, respectively. In addition, $S_3 \times S_3$ and $\mathrm{SmallGroup}(48,15)$ satisfy the hypotheses of Lemma~\ref{lem:gccycnonab}\ref{gccycnonab3}, and their graphs have diameter $2$ and $3$, respectively.

Before presenting examples of groups that satisfy Lemma~\ref{lem:gccycnonab}\ref{gccycnonab1}, we further clarify how Assumption~\ref{assump:22g} relates to this lemma (and to Assumption~\ref{assump:nc}).

\begin{prop}
\label{prop:22grouplink}
Suppose that $G$ satisfies Assumption~\ref{assump:22g}. Then $G$ is $2$-generated and has a normal subgroup $N \not\le Z(G)$, with $G/N$ non-cyclic, $G/C_G(N)$ cyclic, and $C_G(N)$ non-abelian. Thus $G$ satisfies the hypotheses of Lemma~\ref{lem:gccycnonab}, and so $\nc(G)$ is the union of two connected components of diameter $2$.
\end{prop}

\begin{proof}
Let $P$ and $Q$ be as in Assumption~\ref{assump:22g}, and let $R$ be the unique maximal subgroup of $Q$. Theorem~\ref{thm:solmaxclass} shows that $G$ contains a normal maximal subgroup $M := P \times R$, and that $G = \langle x,y \rangle$ for each $x \in P \setminus \Phi(P)$ and generator $y$ for $Q$. Additionally, $N := Z(M) = Z(P) \times R$ is not a subgroup of $Z(G)$ by assumption. Since $Z(P) = \Phi(P) < P$, both $P$ and $M$ are non-abelian. Thus $G/N = G/Z(M)$ is not cyclic, and so $G$ and $N$ are as in Assumption~\ref{assump:nc}. Moreover, $C_G(N) = M$, and hence $G/C_G(N)$ is cyclic. Therefore, $G$ satisfies the hypotheses of Lemma~\ref{lem:gccycnonab}, and the final part of the result follows from Lemma~\ref{lem:gccycnonab}\ref{gccycnonab1}.
\end{proof}

We will call a group $G$ a \emph{$[2,2]$-group} if $\nc(G)$ is the union of two connected components of diameter $2$. The following example describes an infinite family of such groups.

\begin{eg}
\label{eg:suz22}
Let $G$ be the finite simple Suzuki group $\mathrm{Sz}(q)$, where $q:=2^i$ with $i$ an odd integer at least $3$. Additionally, let $P$ be a Sylow $2$-subgroup of $G$, and $N:=N_G(P)$. Then $|P| = q^2$, and $N$ is a maximal subgroup of $G$ isomorphic to the Frobenius group $P \sd C_{q-1}$ \cite[\S4, p.~133, \& Theorem 9]{suzuki}. Given a primitive prime divisor $r$ of $2^i-1$, let $N_r:=P \sd C_r \le N$. Then $N_r$ is also Frobenius, hence $Z(N_r) = 1$. Moreover, each cyclic subgroup of $N_r$ of order $r$ acts irreducibly on $P/\Phi(P)$ \cite[Theorem 3.5]{hering}. Thus $N_r$ satisfies all conditions of Theorem~\ref{thm:solmaxclass}\ref{solmaxclass1}.

We claim that $N_r$ is a $[2,2]$-group. By Proposition~\ref{prop:22grouplink}, it suffices to show that $G$ satisfies Assumption~\ref{assump:22g}. As the unique maximal subgroup of $C_r$ is the trivial subgroup, it remains only to prove that $\Phi(P) = Z(P) \not\le Z(N_r) = 1$. Let $\theta$ be the automorphism $\alpha \mapsto \alpha^{\sqrt{2q}}$ of $\mathbb{F}_{q}$. Then \cite[pp.~111-112 \& Theorem 7]{suzuki} shows that $P$ is isomorphic to the group $\{(\alpha,\beta) \mid \alpha,\beta \in \mathbb{F}_{q}\}$, where $(\alpha,\beta)(\gamma,\delta):=(\alpha+\gamma,\alpha \gamma^\theta+\beta+\delta)$ for all $(\alpha,\beta),(\gamma,\delta) \in P$. 

Now, $Z(P) = \{(0,\beta) \mid \beta \in \mathbb{F}_q\} \not\le 1 = Z(N_r)$. Using the fact that ${\alpha \mapsto \alpha\alpha^{\theta}}$ is an automorphism of $\mathbb{F}_{q}^\times$, we calculate that $Z(P)$ contains $P'$ and is equal to the subgroup $K$ generated by all squares in $P$. Since $\Phi(P) = KP'$, it follows that $\Phi(P) = Z(P)$. Hence $N_r$ is a $[2,2]$-group.
\end{eg}

We observe using Magma that $\mathrm{SmallGroup}(96,3)$ is the unique smallest finite $[2,2]$-group. Additionally, there exist $[2,2]$-groups with odd order, e.g.,  $\mathrm{SmallGroup}(9477,4035)$, and with no Sylow subgroup of prime order, e.g., $\mathrm{SmallGroup}(288,3)$.

The following theorem summarises this section's main results.

\begin{thm}
\label{thm:ncsummary}
Suppose that $G$ contains a normal subgroup $N$, such that $G/N$ is not cyclic and $N \not\le Z(G)$, and let $C:=C_G(N)$. Then one of the following holds.
\begin{enumerate}[label={(\roman*)},font=\upshape]
\item $\nc(G)$ has an isolated vertex, and $\nd(G)$ is connected with diameter $2$. Additionally, $G$ is soluble, $C$ is abelian and maximal in $G$, and each isolated vertex lies in $C \setminus N$.
\item $\nc(G)$ is connected with diameter $2$, $3$ or $4$. If $\diam(\nc(G)) = 4$, then $G$ is infinite, $G/C$ is cyclic, and $C$ is non-abelian.
\item $\nc(G)$ is the union of two connected components of diameter $2$, with one component consisting of the elements of $C \setminus Z(C)$. Moreover, $C$ is non-abelian and maximal in $G$.
\end{enumerate}
Furthermore, if $G$ is finite, then (iii) holds if and only if $G$ satisfies Assumption~\ref{assump:22g}.
\end{thm}

\begin{proof}
We may assume that $G$ is $2$-generated; otherwise, $\diam(\nc(G)) = 2$ by Proposition~\ref{prop:ncbasics}\ref{highgengp}. If $G/C$ is not cyclic, then Lemma~\ref{lem:gcnoncyc} applies, and (ii) holds. Otherwise, either Lemma~\ref{lem:gccycab} or Lemma~\ref{lem:gccycnonab} applies, depending on whether $C$ is abelian. Specifically, if $C$ is abelian, then (i) or (ii) holds, and otherwise, (ii) or (iii) holds. Thus we observe from Lemma~\ref{lem:gccycnonab} and Proposition~\ref{prop:22grouplink} that if $G$ is finite, then (iii) holds if and only if $G$ satisfies Assumption~\ref{assump:22g}.
\end{proof}

The above theorem implies Theorems~\ref{thm:ncmainsummary} and \ref{thm:sumtwo} in the case where $G/Z(G)$ has a proper non-cyclic quotient. Additionally, if $\nc(G)$ has two nontrivial components, then Lemma~\ref{lem:gccycnonab}\ref{gccycnonab1} shows that $K \cap C = Z(C)$ for each maximal subgroup $K$ of $G$ distinct from the normal, non-abelian maximal subgroup $C$. As $Z(G) < Z(C)$ by Lemma~\ref{lem:ninzc}, we can use Lemma~\ref{lem:maxnormconnected} and Proposition~\ref{prop:gkzmprim} to deduce further information about infinite groups in this case. It is also easy to show using Theorem~\ref{thm:ncsummary} that if a group $G$ contains non-central normal subgroups $N_1$ and $N_2$, with $G/N_1$ and $G/N_2$ non-cyclic and $C_G(N_1) \ne C_G(N_2)$, then $\nc(G)$ is connected.

Note that Theorem~\ref{thm:ncsummary} applies whenever $G$ is a free product $\langle a \rangle * \langle b \rangle$ of nontrivial cyclic groups, e.g., with $N = \langle (ab)^k \rangle$ for some $k \ge 2$ and $C = \langle ab \rangle$. By appealing to the nature of $G$ as a free product, we prove in \cite[\S5.10]{saulthesis} that $\diam(\nc(G)) = 2$ unless $|a| = |b| = 2$, in which case $G$ is the infinite dihedral group, $\nc(G) \setminus \nd(G) = \{ab,ba\}$, and $\diam(\nd(G)) = 2$.

\section{Groups with each proper quotient cyclic}
\label{sec:finitesol}

In order to prove Theorems~\ref{thm:ncmainsummary} and \ref{thm:sumtwo}, it remains to consider the case where every proper quotient of $G/Z(G)$ is cyclic. As in the previous section, we will implicitly use Proposition~\ref{prop:ncbasics}\ref{nodiamonenc}, which states that each nontrivial component of $\nc(G)$ has diameter at least $2$. 

The following lemma generalises the classification given in \cite[\S3]{lmrd} of \emph{finite} groups whose proper quotients are all cyclic. Here, by a \emph{central extension} of $G$, we mean a group $H$ such that $H/Z(H) \cong G$. As above, an abstract group is primitive if it has a core-free maximal subgroup, which is a point stabiliser for the corresponding primitive action.

\begin{lem}
\label{lem:allquocyc}
Suppose that each proper quotient of $G$ is cyclic. Then one of the following holds:
\begin{enumerate}[label={(\roman*)},font=\upshape]
\item for each central extension $H$ of $G$ (including $G$ itself), every maximal subgroup of $H$ is normal in $H$, and hence $G$ is not primitive;
\item $G$ is soluble and primitive with a (unique) minimal normal subgroup and a cyclic point stabiliser; or
\item $G$ is insoluble and primitive, and $C_G(N) = 1$ for each normal subgroup $N \ne 1$ of $G$.
\end{enumerate}
\end{lem}

\begin{proof}
We split the proof into three cases, which together account for all possibilities.

\medskip

\noindent \textbf{Case (a)}: $G$ is not primitive. Suppose that $G$ contains a maximal subgroup $M$, and let $J:= \mathrm{Core}_G(M)$. Then $J \ne 1$, and so $G/J$ is cyclic. Thus $M/J \trianglelefteq G/J$, and so $M \trianglelefteq G$. As each maximal subgroup of $H$ not containing $Z(H)$ is normal, (i) follows.

\medskip

\noindent \textbf{Case (b)}: $G$ is primitive, and $C_G(J) \ne 1$ for some nontrivial normal subgroup $J$ of $G$. Then $N:=C_G(J)$ is a minimal normal subgroup of $G$. If $G$ contains a distinct minimal normal subgroup $K$, then $K$ is non-abelian and equal to $K/(K \cap N) \cong NK/N$. In particular, $NK/N$ is not cyclic, and so neither is $G/N$, a contradiction. Hence $N$ is the unique minimal normal subgroup of $G$. Moreover, since $C_G(N)$ contains the nontrivial subgroup $J$ (which is now clearly equal to $N$), it follows that $N$ is abelian. As $G/N$ is cyclic, $G$ is soluble, and (ii) holds.

\medskip

\noindent \textbf{Case (c)}: $G$ is primitive, and $C_G(N) = 1$ for each nontrivial normal subgroup $N$ of $G$. If $G$ is soluble, then the penultimate subgroup $R$ in the derived series of $G$ is nontrivial and abelian, and hence $C_G(R) \ne 1$. However, $R \trianglelefteq G$, a contradiction. Thus (iii) holds.
\end{proof}

If case (i) of Lemma~\ref{lem:allquocyc} holds, then Theorem~\ref{thm:nilpncgeneral} applies. Hence it remains to consider cases (ii) and (iii). We will write $\overline H:=H/Z(G)$ when $H$ is a subgroup of $G$ containing $Z(G)$.

\begin{prop}
\label{prop:primsolquomax}
Suppose that $\overline G$ is soluble and primitive with every proper quotient cyclic, and let $L$ be a non-abelian maximal subgroup of $G$. Then $L \trianglelefteq G$ and $Z(L) \le Z(G)$.
\end{prop}

\begin{proof}
It follows from Lemma~\ref{lem:allquocyc} that $\overline G$ contains a unique minimal normal subgroup $\overline N$, and a cyclic point stabiliser $\overline M$ such that $\overline G = \overline N \sd \overline M$. Suppose first that $Z(G) \not\le L$. Then $L \trianglelefteq G$. Additionally, there is no $x \in G$ satisfying $L = C_G(x)$, and thus $Z(L) < Z(G)$.

Assume from now on that $Z(G) \le L$, and note that $\overline L$ is a non-cyclic maximal subgroup of $\overline G$. Thus $\overline L$ is not a complement of $\overline N$ in $\overline G$, and so $\overline L$ is not core-free in $\overline G$. Hence $\overline N \le \overline L$. Moreover, as $\overline G/\overline N$ is cyclic, we see that $\overline L/\overline N \trianglelefteq \overline G/\overline N$, and it follows that $L \trianglelefteq G$.

Next, $\overline G/\overline{Z(L)}$ is isomorphic to $G/Z(L)$, which is not cyclic, and so $\overline N \not\le \overline{Z(L)}$. As $\overline N$ lies in each nontrivial normal subgroup of $\overline G$, we conclude that $\overline{Z(L)} = 1$, and so $Z(L) = Z(G)$.
\end{proof}

\begin{lem}
\label{lem:ncprimsolquo}
Suppose that $\overline G$ is soluble and primitive with every proper quotient cyclic, and that $\nc(G)$ has an edge. Then $\nc(\overline G)$ has isolated vertices, and $\diam(\nd(G)) = 2$.
\end{lem}

\begin{proof}
Since the primitive group $\overline G$ has a minimal normal subgroup by Lemma~\ref{lem:allquocyc}, it is clear that each point stabiliser $\overline K$ of $\overline G$ is cyclic. Additionally, $Z(\overline G) = 1$, and since $\overline G = \langle k, g \rangle$ whenever $k$ is a generator for $\overline K$ and $g \in \overline G \setminus \overline K$, every such $k$ is an isolated vertex of $\nc(\overline G)$.

Now, let $x,y \in \nd(G)$, and assume that $G$ is $2$-generated; else $\diam(\nc(G)) = 2$ by Proposition~\ref{prop:ncbasics}\ref{highgengp}. By Propositions~\ref{prop:ncbasics}\ref{isolvertnc} and~\ref{prop:primsolquomax}, there exist normal maximal subgroups $L$ and $M$ of $G$ with $x \in L \setminus Z(L)$, $y \in M \setminus Z(M)$ and $Z(L),Z(M) \le Z(G)$. By Lemma~\ref{lem:ncmaxnormcombined}, $d(x,y) \le 2$.
\end{proof}

As we mentioned in \S\ref{sec:intro}, the generating graph of a finite group is connected precisely when all proper quotients of that group are cyclic \cite[Theorem 1 \& Corollary 2]{burnessspread}. However, this is not the case for the non-commuting, non-generating graph. Indeed, $G:=\mathrm{AGL}(1,5) = \overline G$ satisfies the hypotheses of Lemma~\ref{lem:ncprimsolquo}, and so $\nd(G) \ne \nc(G)$. On the other hand, Magma computations show that $H:=D_{10} \sd C_8$ is a non-split extension of $Z(H)$ by $G$, and $\diam(\nc(H)) = 2$.

In the following theorem, we assume that $G$ itself is primitive, so that $Z(G) = 1$.

\begin{lem}
\label{lem:ncimprimcyc}
Suppose that $G$ is insoluble, non-simple and primitive with every proper quotient cyclic. Then $\nd(G)$ is connected with diameter $2$ or $3$. Moreover, if $\nc(G)$ has an isolated vertex $r$, then $G$ is infinite and 
each proper subgroup of $G$ containing $r$ is core-free. Finally, if $G$ is $2$-generated, then it contains a normal maximal subgroup with trivial centre.
\end{lem}

\begin{proof}
We may again assume that $G$ is $2$-generated, else $\diam(\nc(G)) = 2$ by Proposition~\ref{prop:ncbasics}\ref{highgengp}. Let $N$ be a nontrivial proper normal subgroup of $G$, and $M$ a maximal subgroup containing $N$. Then for each overgroup $H$ of $N$ in $G$, the cyclic group $G/N$ normalises $H/N$, and hence $H \trianglelefteq G$. Thus each non-normal subgroup of $G$ is core-free. In particular, $M \trianglelefteq G$, and $Z(M) = 1$ by Lemma~\ref{lem:allquocyc}. Since no maximal subgroup of a finite insoluble group is abelian \cite{herstein}, the statement about an isolated vertex follows from Propositions~\ref{prop:ncbasics}\ref{isolvertnc} and~\ref{prop:isolabmax}. 

Now, let $x$ and $y$ be non-isolated vertices of $\nc(G)$. Then Proposition~\ref{prop:ncbasics}\ref{isolvertnc} shows that $x \in K \setminus Z(K)$ and $y \in L \setminus Z(L)$ for some maximal subgroups $K$ and $L$ of $G$. We may assume that $K \ne L$, as otherwise $d(x,y) \le 2$ by Proposition~\ref{prop:ncbasics}\ref{propernc}. Observe also that $K \cap M$ and $L \cap M$ are nontrivial maximal subgroups of $K$ and $L$, respectively.

Suppose first that $K \trianglelefteq G$, so that $Z(K) = 1$ by this proof's first paragraph. We may assume that $y \notin K$, as otherwise we could set $L = K$. Applying Lemma~\ref{lem:centremaxmax}\ref{centremaxmax1} to $(y,L,K)$ yields $y \sim h$ for some $h \in K \cap L \setminus C_{K \cap L}(y)$. As $d(x,h) \le 2$ by Proposition~\ref{prop:ncbasics}\ref{propernc}, we obtain $d(x,y) \le 3$.

By symmetry, we may assume from now on that neither $K$ nor $L$ is normal in $G$, i.e., that both are core-free in $G$, and that $x,y \notin M$. Then the nontrivial subgroups $K \cap M$ and $L \cap M$ are not normal in $G$. Since $K \cap M \trianglelefteq K$ and $\langle K, L \rangle = G$, it follows that $L$ does not centralise $K \cap M$. Additionally, applying Lemma~\ref{lem:centremaxmax}\ref{centremaxmax1} to $(x,K,M)$ gives $C_{K \cap M}(x) < K \cap M$. Thus if $K \cap M \le L$, then there exists an element $a \in K \cap M$ that centralises neither $x$ nor $L$. Hence $x \sim a$, and Proposition~\ref{prop:ncbasics}\ref{propernc} yields $d(a,y) \le 2$. Therefore, $d(x,y) \le 3$.

If instead $K \cap M \not\le L$, then $\langle K \cap M, L \rangle = G$. As $L \cap M$ is normalised by $L$ but not $G$, we deduce that $K \cap M \not \le C_G(L \cap M)$. Since $C_{K \cap M}(x) < K \cap M$, and similarly $C_{L \cap M}(y) < L \cap M$, it follows that there exists an element $b \in K \cap M$ that centralises neither $L \cap M$ nor $x$, and an element $c \in L \cap M$ that centralises neither $b$ nor $y$. Thus $x \sim b \sim c \sim y$ and $d(x,y) \le 3$.
\end{proof}

Using Magma, we see that $G_1:=A_5 \wr C_2$ and $G_2:=(A_5 \times A_5) \sd C_4$ (with a point stabiliser of index $25$) satisfy the hypotheses of Lemma~\ref{lem:ncimprimcyc}, with $\diam(\nc(G_1)) = 2$ and $\diam(\nc(G_2)) = 3$.

\begin{lem}
\label{lem:ncinsprimquo}
Suppose that $\overline G$ is insoluble, non-simple and primitive with every proper quotient cyclic.
\begin{enumerate}[label={(\roman*)},font=\upshape]
\item \label{ncinsprimquo1} The subgraph $X$ of $\nc(G)$ induced by the vertices in $\{g \in G \setminus Z(G) \mid Z(G)g \in \nd(\overline{G})\}$ has diameter at most $k:=\diam(\nd(\overline{G})) \in \{2,3\}$. Hence if $\nc(\overline G)$ has no isolated vertices, and in particular if $\overline G$ is finite, then $\diam(\nc(G)) \le k$.
\item If $X \ne \nd(G)$, then $\nd(G)$ is connected with diameter at most $4$.
\end{enumerate}
\end{lem}

\begin{proof}
As $Z(\overline G) = 1$, (i) follows from Lemma~\ref{lem:ncimprimcyc}, and Lemma~\ref{lem:quotientnc}\ref{quotientnc2} with $N = Z(G)$.

To prove (ii), we may assume that $X \ne \nd(G)$ and (by Proposition~\ref{prop:ncbasics}\ref{highgengp}) that $G$ is $2$-generated. Then $\overline G$ is also $2$-generated, and Lemma~\ref{lem:ncimprimcyc} implies that $G$ contains a normal maximal subgroup $M$ with $Z(M) = Z(G)$, and that each element of $M \setminus Z(G)$ lies in $X$.

Let $y,y' \in \nd(G) \setminus X$, so that $y,y' \notin M$. Then Proposition~\ref{prop:ncbasics}\ref{isolvertnc} shows that $y \in K \setminus Z(K)$ for some maximal subgroup $K$ of $G$, and applying Lemma~\ref{lem:centremaxmax}\ref{centremaxmax1} to $(y,K,M)$ yields $y \sim m$ for some $m \in K \cap M$. Similarly, $y' \sim m'$ for some $m' \in M$. Proposition~\ref{prop:ncbasics}\ref{propernc} gives $d(m,m') \le 2$, and so $d(y,y') \le d(y,m)+d(m,m')+d(m',y) \le 4$.

By (i), it remains to consider $d(y,x)$, with $x \in X$. We also observe from (i) that $d(m,x) \le 3$. Hence $d(y,x) \le d(y,m) + d(m,x) \le 4$, and we conclude that $\diam(\nd(G)) \le 4$.
\end{proof}

Recall from above that if $G = \overline G = (A_5 \times A_5) \sd C_4$, then $\diam(\nc(G)) = 3$. We can use Magma to show that the non-commuting, non-generating graphs of the central extensions $G \times C_2$ and $G \times C_3$ of $G$ are connected with diameter $2$ and $3$, respectively (cf. \cite[Proposition 16]{nilppaper}).

\begin{ques}
\label{ques:imprimnc}
Does there exist an infinite group $G$ such that $\overline G$ satisfies the hypotheses of Lemma~\ref{lem:ncinsprimquo} and either $\diam(\nd(G)) = 4$, or $\diam(\nd(G)) = 3$ and $\nd(G) \ne \nc(G)$?
\end{ques}

In the following example, we prove that $\nc(G_k)$ is connected (and hence has no isolated vertices) for a certain infinite family of infinite groups $G_k$ that satisfy the hypotheses of Lemma~\ref{lem:ncimprimcyc}. 

\begin{eg}
\label{eg:houghtonsubgroups}
For each positive integer $k$, let $G_k$ be the group of permutations of $\mathbb{Z}$ generated by the simple alternating group $\mathrm{Alt}(\mathbb{Z})$ and the translation $t_k$ that maps $x$ to $x+k$ for all $x \in \mathbb{Z}$. Then $\mathrm{Alt}(\mathbb{Z})$ is the unique minimal normal subgroup of the insoluble group $G_k$ \cite[Proposition 2.5]{coxrinf}. Moreover, $G_k$ is $2$-generated, and every proper quotient of $G_k$ is cyclic \cite[Theorem 4.1]{cox}. Observe that any maximal subgroup of $G_k$ containing $t_k$ is core-free, and so $G_k$ is primitive.

Assume now that $k \ge 3$, and let $g \in G_k \setminus \{1\}$. Since $\mathrm{Alt}(\mathbb{Z})$ is generated by its $3$-cycles and $C_{G_k}(\mathrm{Alt}(\mathbb{Z})) = 1$, it follows that there exists a $3$-cycle $\alpha \in \mathrm{Alt}(\mathbb{Z})$ with $[g,\alpha] \ne 1$. Furthermore, the proofs of \cite[Lemmas 4.2--4.3]{cox} show that no $3$-cycle in $\mathrm{Alt}(\mathbb{Z})$ lies in a generating pair for $G_k$. Hence $g \sim \alpha$. In particular, $\nc(G_k) = \nd(G_k)$. Using Lemmas~\ref{lem:ncimprimcyc} and~\ref{lem:ncinsprimquo}\ref{ncinsprimquo1}, we conclude that $\nc(G_k)$ is connected with diameter $2$ or $3$, as is $\nc(H)$ for each central extension $H$ of $G_k$.
\end{eg}

It would be interesting to determine $\diam(\nc(G_k))$ precisely for each $k \ge 3$, and to investigate $\nc(G_k)$ when $k \in \{1,2\}$, where each $g \in G_k \setminus \{1\}$ lies in a generating pair \cite[Theorem 6.1]{cox}.

We now prove this paper's main theorems.

\begin{proof}[Proof of Theorems~\ref{thm:ncmainsummary} and \ref{thm:sumtwo}]
If $\overline G = G/Z(G)$ has a proper non-cyclic quotient, then $G$ contains a normal subgroup $N$ such that $Z(G) < N$ and $G/N$ is not cyclic. Thus in this case Theorem~\ref{thm:ncsummary} applies, and case (i), (iii), (iv) or (v) of Theorem~\ref{thm:ncmainsummary} holds. Otherwise, one of the three cases in Lemma~\ref{lem:allquocyc} applies, with $\overline G$ in place of $G$. If every maximal subgroup of $G$ is normal, or if $\overline G$ is soluble and primitive, then we can use Theorem~\ref{thm:nilpncgeneral} or Lemma~\ref{lem:ncprimsolquo}, respectively, to show that case (i) or (iii) of Theorem~\ref{thm:ncmainsummary} holds. If instead $\overline G$ is insoluble and primitive, then Lemma~\ref{lem:ncinsprimquo} shows that case (ii) or (iii) holds. This completes the proof of Theorem~\ref{thm:ncmainsummary}.

Assume now that $G$ is finite. By the previous paragraph, if $\nc(G)$ is the union of two connected components of diameter $2$, then Theorem~\ref{thm:ncsummary} applies. That theorem and Proposition~\ref{prop:22grouplink} imply that $\nc(G)$ is such a union if and only if $G$ satisfies Assumption~\ref{assump:22g}, and Theorem~\ref{thm:sumtwo} follows.
\end{proof}

We see from the above proof that if $\nc(G)$ has two nontrivial connected components, then Theorem~\ref{thm:ncsummary} applies. Hence, as discussed below the proof of that theorem, Lemma~\ref{lem:maxnormconnected} and Proposition~\ref{prop:gkzmprim} yield further information about the structures of infinite groups in this case.

\subsection*{Acknowledgments}
This work was supported by the University of St Andrews (St Leonard's International Doctoral Fees Scholarship \& School of Mathematics and Statistics PhD Funding Scholarship), and by EPSRC grant number EP/W522422/1. The author is grateful to Jendrik Brachter for the proof of Proposition~\ref{prop:isolabmax} in the case $M \not\trianglelefteq G$; to Colva Roney-Dougal, Peter Cameron, Martyn Quick and Donna Testerman for helpful discussions regarding the original thesis on which this work is based; and to Colva and an anonymous referee for useful comments on this paper.

\bibliographystyle{abbrv}
\bibliography{Nonsimplerefs}

\end{document}